\documentclass[12pt,a4paper]{amsart}
\usepackage[utf8]{inputenc}
\usepackage{amsfonts}
\usepackage{amssymb}
\usepackage{amsmath}
\usepackage{amsthm}
\usepackage{cite}
\usepackage{enumitem}
\usepackage{tikz}
\usepackage{subfigure}
\usepackage{dsfont}

\theoremstyle{plain}
\newtheorem{thm}{Theorem}

\newtheorem{lem}{Lemma}
\newtheorem{prop}{Proposition}
\newtheorem{cor}{Corollary}
\newtheorem{rmk}{Remark}

\newcommand{\svec}[2]{
\ensuremath{
\begin{pmatrix}
#1 \\ #2 \\
\end{pmatrix}}}

\providecommand{\N}{\mathbb{N}}
\providecommand{\R}{\mathbb{R}}
\providecommand{\Z}{\mathbb{Z}}
\providecommand{\C}{\mathbb{C}}
\providecommand{\Q}{\mathbb{Q}}
\providecommand{\eps}{\varepsilon}

\newcommand{\norm}[1]{\left\lVert #1 \right\rVert}

\newcommand{\brac}[1]{\left\{ #1 \right\}}
\renewcommand{\i}{\mathrm{i}}
\renewcommand{\L}[1]{L^{#1}(\R^N)}
\newcommand{\LL}[1]{L^{#1}(\R^N) \times L^{#1}(\R^N)}

\newcommand{\step}[2]{\vspace*{0.2cm} \underline{Step #1}: \textit{#2} \vspace*{0.1cm}}

\newcommand{\myref}{Ev\'{e}quoz and Weth }
\newcommand{\myre}{Ev\'{e}quoz }
\newcommand{\absatz}{\vspace*{0.5cm}}

\setitemize{itemsep=+2pt}
\setenumerate{itemsep=+2pt}
\begin{document}

\allowdisplaybreaks

\title{Dual Variational Methods for a nonlinear Helmholtz system}

\author{Rainer Mandel and Dominic Scheider}
\address{R. Mandel\hfill\break
Karlsruhe Institute of Technology \hfill\break
Institute for Analysis \hfill\break
Englerstra{\ss}e 2 \hfill\break
D-76131 Karlsruhe, Germany}
\email{rainer.mandel@kit.edu}
\address{D. Scheider\hfill\break
Karlsruhe Institute of Technology \hfill\break
Institute for Analysis \hfill\break
Englerstra{\ss}e 2 \hfill\break
D-76131 Karlsruhe, Germany}
\email{dominic.scheider@kit.edu}
\date{\today}


\begin{abstract}
	This paper considers a pair of coupled nonlinear Helmholtz equations
	\begin{align*}
	\begin{cases}
		-\Delta u - \mu u = a(x) \left( |u|^\frac{p}{2} + b(x) |v|^\frac{p}{2}  \right)
		|u|^{\frac{p}{2} - 2}u,
		\\
		-\Delta v - \nu v = a(x) \left( |v|^\frac{p}{2} + b(x) |u|^\frac{p}{2}  \right)
		|v|^{\frac{p}{2} - 2}v
	\end{cases}
	\end{align*}
	on $\R^N$ where $\frac{2(N+1)}{N-1} < p < 2^\ast$.
	The existence of nontrivial strong solutions in $W^{2, p}(\R^N)$ is established
	using dual variational methods. 
	The focus lies on necessary and sufficient conditions on the parameters 
	deciding whether or not both components of such solutions are nontrivial. 
\end{abstract}

\maketitle

\allowdisplaybreaks

\section{Introduction and Main Results}

In this paper we study a system of two coupled nonlinear Helmholtz equations as it arises in models of nonlinear optics. 
More specifically, we intend to find a pair of fully nontrivial, real-valued and strong solutions $(u, v) \in W^{2,p}(\R^N) \times W^{2,p}(\R^N)$, $N \geq 2$, of 
\begin{align}\label{eq_system}
	&\begin{cases}
		- \Delta u - \mu u = 
		a(x) \left( |u|^\frac{p}{2} + b(x) |v|^\frac{p}{2}  \right)
		|u|^{\frac{p}{2} - 2}u 						& \text{on } \R^N,
		\\
		- \Delta v - \nu v \, = a(x) \left( |v|^\frac{p}{2} + b(x) |u|^\frac{p}{2}  \right)
		|v|^{\frac{p}{2} - 2}v					  		& \text{on } \R^N,
		\\
		u, v \in L^p(\R^N)
	\end{cases}
\end{align}
with $\mu, \nu > 0$, $\frac{2(N+1)}{N-1} < p < 2^\ast$ and nonnegative, $\Z^N$-periodic coefficients $a, b \in L^\infty(\R^N)$.
Here $2^\ast$ denotes the critical Sobolev exponent, $2^\ast = \frac{2N}{N-2}$ for $N \geq 3$, $2^\ast = \infty$ for $N = 2$.
A solution $(u, v)$ of~\eqref{eq_system} is said to be semitrivial if either $u = 0$ or $v = 0$ and fully nontrivial if both $u \neq 0$ and $v \neq 0$.
Our aim is to find necessary and sufficient conditions for the existence of fully nontrivial solutions of~\eqref{eq_system}. 

\absatz

To our knowledge, systems of Helmholtz equations have not been discussed in the literature so far in contrast to the Schrödinger case where $\mu, \nu < 0$ in~\eqref{eq_system}; for a comparison with our results, we refer to the end of this introduction.
Likewise, there is much literature on nonlinear Schrödinger equations of the form
\begin{equation*}
	- \Delta w + \lambda w = Q(x) |w|^{p-2} w
	\qquad \text{on } \R^N
\end{equation*}
with some $\lambda > 0$;
for instance, Szulkin and Weth prove the existence of ground state solutions in the Sobolev space $H^1(\R^N)$ in Theorem~1.1 of~\cite{WethSzulkin} by constraint minimization on the associated Nehari set.
In contrast, the corresponding Helmholtz problem
\begin{equation}\label{eq_helmholtz}
	- \Delta w - \lambda w = Q(x) |w|^{p-2} w
	\qquad \text{on } \R^N
\end{equation}
has only been discussed during the past five years.
Since $\lambda > 0$ belongs to the essential spectrum of $- \Delta$, the Helmholtz case requires different concepts in order to handle oscillating solutions with slow decay which, in general, are not elements of $H^1(\R^N)$. In the radial case, such oscillation and decay properties are studied in~\cite{radial}.
\myref discuss the case of compactly supported $Q$ and $2 < p < 2^\ast$ in~\cite{EvequozWeth_real, EvequozWeth_branch}. In~\cite{EvequozWeth_real}, they study an exterior problem where the nonlinearity vanishes and knowledge about the far-field expansion of solutions is available. The remaining problem on a bounded domain is solved by variational techniques.
The approach in~\cite{EvequozWeth_branch} uses Leray-Schauder continuation with respect to the parameter $\lambda$ in order to find solutions of~\eqref{eq_helmholtz}.
We will follow the ideas of \myref presented in~\cite{Evequoz_plane, EvequozWeth}. They introduce a dual variational approach, transforming the Helmholtz equation~\eqref{eq_helmholtz} into
\begin{align*}
	|\bar{w}|^{p'-2}\bar{w} 
	= Q(x)^\frac{1}{p} (- \Delta - \lambda)^{-1} \left[ Q(x)^\frac{1}{p} \bar{w} \right]
	\quad
	\text{where }
	\bar{w}(x) = Q(x)^\frac{1}{p'} |w(x)|^{p-2} w(x).
\end{align*}
Here, the resolvent-type operator $(- \Delta - \lambda)^{-1}$ is obtained by the Limiting Absorption Principle of Guti\'{e}rrez, see the explanations before Theorem~6 in~\cite{Gutierrez}.
The resulting dual equation in $\L{p'}$ is variational; using the Mountain Pass Theorem, the authors prove the existence of a ground state $\bar{w} \in \L{p'}$ of the dual problem, which yields a strong solution  $w \in W^{2,q}(\R^N) \cap C^{1,\alpha}(\R^N)$, $q \in [p, \infty)$ and $\alpha \in (0, 1)$, of the Helmholtz equation~\eqref{eq_helmholtz}. Here  $\frac{2(N+1)}{N-1} < p < 2^\ast$ and $Q$ is assumed to be positive and either periodic or decaying at infinity. 
In the latter case, it is shown that infinitely many solutions exist; for periodic $Q$, \myre proves the corresponding statement in~\cite{Evequoz_periodic}. He also shows that the dual problem possesses a gound state if $Q$ is assumed to be the sum of a periodic and a decaying term.
\myre further generalizes these results in~\cite{Evequoz_orlicz}; for instance, it is shown that the dual variational techniques apply for any $p \in (2, 2^\ast)$ if $Q$ satisfies suitable integrability conditions. 
In~\cite{Evequoz_critical}, \myre and Ye\c{s}il prove the existence of a dual ground state in the critical case $p = 2^\ast$ for $N \geq 4$ and the non-existence for $N = 3$ where, again, $Q$ is assumed to be the sum of a decaying and a periodic term.
For continuous, nonnegative $Q$ and $\frac{2(N+1)}{N-1} < p < 2^\ast$, \myre proves existence, concentration and multiplicity of ground states of the dual problem in the high-frequency limit $\lambda \nearrow \infty$ in~\cite{Evequoz_high} based on a comparison of energies with a suitable limit problem.

\absatz

We will show that, under suitable assumptions on the coefficients $a$ and $b$, the dual variational approach and the existence results by \myref in~\cite{Evequoz_plane, EvequozWeth} extend to the case of the system~\eqref{eq_system}. To this end, we will introduce a dual formulation for the system~\eqref{eq_system} of the form
\begin{align}\label{eq_system_dual}
	\begin{cases}
		\partial_{\bar{s}} h(x, \bar{u}, \bar{v})
		= \Psi_\mu \ast \bar{u} 			
		& \text{on } \R^N,
		\\
		\partial_{\bar{t}} h(x, \bar{u}, \bar{v})
		= \Psi_\nu \ast \bar{v}		
		& \text{on } \R^N,
		\vspace*{0.2cm}
		\\
		\bar{u}, \bar{v} \in L^{p^\prime}(\R^N).
	\end{cases}
\end{align}
In the following section, the role of the convolution kernels
$\Psi_\mu, \Psi_\nu$, see also equations~(11),~(45) of~\cite{EvequozWeth}, will be explained. So will be the transformation involving $\bar{u}, \bar{v} \in L^{p^\prime}(\R^N)$,
\begin{equation}\label{eq_transform-ubaru}
\begin{split}
	\bar{u}(x) &:=a(x) \left( |u(x)|^\frac{p}{2} + b(x) |v(x)|^\frac{p}{2}  \right)
		|u(x)|^{\frac{p}{2} - 2} u(x), 
	\\
	\bar{v}(x) &:= a(x) \left( |v(x)|^\frac{p}{2} + b(x) |u(x)|^\frac{p}{2}  \right)
		|v(x)|^{\frac{p}{2} - 2} v(x)
\end{split}
\end{equation}
and a suitable function $h: \R^N \times \R \times \R \to \R$, cf. Proposition~\ref{prop_fbhb} below.
Notice that we use the notation $\bar{u}, \bar{v} \in L^{p^\prime}(\R^N)$ in place of $u, v \in L^{p}(\R^N)$ whenever we are working in the dual setting; it does not denote complex conjugation, which does not occur in this paper. 

\absatz

The dual system~\eqref{eq_system_dual} is variational; we introduce the corresponding energy functional
\begin{equation}\label{eq_Jmunub}
\begin{alignedat}{2}
	& J_{\mu\nu}: \quad 
	&& L^{p^\prime}(\R^N) \times L^{p^\prime}(\R^N) \to \R,
	\\
	& J_{\mu\nu}(\bar{u}, \bar{v})
	&& := 
	\int_{\R^N} h(x, \bar{u}, \bar{v}) \: \mathrm{d}x
		- \frac{1}{2} \int_{\R^N} \bar{u} \Psi_\mu \ast \bar{u} + \bar{v} \Psi_\nu \ast \bar{v} \: \mathrm{d}x
\end{alignedat}
\end{equation}
with mountain pass level
\begin{align}\label{eq_mplevel}
\begin{split}
		&c_{\mu\nu} := \inf_{\gamma\in\Gamma_{\mu\nu}} \sup_{0 \leq t \leq 1} 
		J_{\mu\nu}(\gamma(t))
		\\
		&\text{where }
		\Gamma_{\mu\nu} 
		:= \brac{\gamma \in C([0, 1], L^{p^\prime}(\R^N) \times L^{p^\prime}(\R^N)): \: \gamma(0) = 0, 
		J_{\mu\nu} (\gamma(1)) < 0}.
\end{split}
\end{align}
For the definition of $h$, we refer to Proposition~\ref{prop_fbhb}.
The main results will be proved under the following assumptions:
\begin{equation}\label{eq_assumptions}
\begin{gathered}
	N \geq 2, \quad
	\mu, \nu > 0, \quad
	\frac{2(N+1)}{N-1} < p < 2^\ast, 
	\\
	a, b \in L^\infty(\R^N) \text{ are }  [0, 1]^N\text{-periodic with }
	0 \leq b(x) \leq p-1, \: a(x) \geq a_0 > 0.
\end{gathered}
\end{equation}
We denote by $a_-, b_-$ the (essential) infimum and by $a_+, b_+$ the (essential) supremum of the functions $a$ and $b$, respectively.

\begin{thm}[Existence Theorem]\label{thm_existence}
	Assuming~\eqref{eq_assumptions},
	there exists a nontrivial critical point $(\bar{u}, \bar{v}) \in \LL{p^\prime}$ 
	of the functional $J_{\mu\nu}$ on the mountain pass level $c_{\mu\nu} > 0$  
	and $(u, v) := \nabla_{\bar{s}, \bar{t}} h(\,\cdot\, , \bar{u}, \bar{v})$ 
	is a strong solution of~\eqref{eq_system} with 
	$u, v \in W^{2,q}(\R^N) \cap C^{1, \alpha}(\R^N)$ 
	for all $q \in [p, \infty)$ and $\alpha \in (0, 1)$.
\end{thm}

\begin{rmk}\label{rmk_scalar}
Consider the scalar functional
$I_\mu: \: L^{p^\prime}(\R^N) \to \R$ given by $I_\mu(\bar{u}) := J_{\mu\nu}(\bar{u}, 0)$.
Lemma~\ref{lem_prop_hb} (d) below provides the formula
\begin{equation}\label{eq_Ilambda}
	I_{\mu} (\bar{u}) 
	= \frac{1}{p^\prime} \int_{\R^N} a(x)^{1-p^\prime} |\bar{u}|^{p^\prime} \: \mathrm{d}x
	- \frac{1}{2} \int_{\R^N} \bar{u} \Psi_\mu \ast \bar{u} \: \mathrm{d}x.
\end{equation}
The results by \myref in~\cite{Evequoz_plane, EvequozWeth}
yield a critical point $\bar{u}$ of $I_\mu$ at the scalar mountain pass level $c_\mu$ and a corresponding solution $u = a^{1 - p'} |\bar{u}|^{p' - 2} \bar{u} \in W^{2,q}(\R^N) \cap C^{1, \alpha}(\R^N)$, $q \in [p, \infty)$ and $\alpha \in (0, 1)$, of the scalar Helmholtz equation
\begin{equation}\label{eq_helmholtz_scal}
	- \Delta u - \mu u = a(x) |u|^{p-2} u
	\quad
	\text{on } \R^N.
\end{equation}
\end{rmk}

Any critical point of the functional $J_{\mu\nu}$ on the level $c_{\mu\nu}$ will henceforth be referred to as a dual ground state of $J_{\mu\nu}$. Notice that Theorem~\ref{thm_existence} yields the existence of a nontrivial dual ground state $(\bar{u}, \bar{v})$ of~\eqref{eq_system_dual}; it does not exclude the semitrivial case where either $\bar{u} = 0$ or $\bar{v} = 0$. Such a semitrivial dual ground state corresponds to a solution of the scalar problem; thus we now discuss under which conditions we find fully nontrivial dual ground states of the system~\eqref{eq_system_dual}. This amounts to statements about the occurrence of fully nontrivial solutions of~\eqref{eq_system}. 
Indeed, for pairs $(\bar{u}, \bar{v}) \in \LL{p'}$ and $(u, v) \in \LL{p}$ 
satisfying~\eqref{eq_transform-ubaru}, a short calculation shows 
that $(u, v)$ is semitrivial (resp. fully nontrivial) if and only if
$(\bar{u}, \bar{v})$ is semitrivial (resp. fully nontrivial). 

\begin{thm}\label{thm_c_mumu_plus-1}
Assume conditions~\eqref{eq_assumptions} to hold.
\begin{itemize}
\item[(a)] If $\quad 2 < p < 4		\quad \text{and} \quad		b_- > 0, \quad$
then every dual ground state of the functional $J_{\mu\nu}$ is fully nontrivial.
\item[(b)] If $\quad p \geq 4	\quad \text{and}\quad	b_- > \frac{a_+}{a_-} \: 2^\frac{p-2}{2} - 1,
\quad$
then there exists $\delta > 0$ with the property that, for $\mu, \nu > 0$ with $\left| \sqrt{\frac{\mu}{\nu}} - 1 \right| < \delta$, every dual ground state of $J_{\mu\nu}$ is fully nontrivial.
\end{itemize}
\end{thm}
\begin{thm}\label{thm_c_mumu_plus-2}
	Assume~\eqref{eq_assumptions} as well as
	\begin{align*}
		p \geq 4		\quad \text{and} \quad		0 \leq b_+ < 2^{\frac{p-2}{2}} - 1.
	\end{align*}
	Then every dual ground state of the functional $J_{\mu\nu}$ is semitrivial.
\end{thm}

In the special case of constant coefficients $a, b$ and $\mu = \nu$ we provide a full characterization of the parameter ranges where semitrivial and fully nontrivial dual ground state solutions occur:
\begin{cor}\label{cor_plus}
Assume that conditions~\eqref{eq_assumptions} hold with constant coefficients $a(x) \equiv a > 0$ and
$b(x) \equiv b \in [0, p-1]$. Then we have the following: 
\begin{itemize}
\item[(a)] $J_{\mu\mu}$ attains the level $c_{\mu\mu}$
  in a fully nontrivial dual ground state if and only if
\begin{align*}
	2 < p < 4 \text{ and } b > 0
	\qquad \text{or} \qquad
	p \geq 4 \text{ and } b \geq 2^{\frac{p-2}{2}} - 1.
\end{align*}
\item[(b)] $J_{\mu\mu}$ attains the level $c_{\mu\mu}$
 in a semitrivial dual ground state if and only if
\begin{align*}
	2 < p < 4 \text{ and } b = 0
	\qquad \text{or} \qquad
	p \geq 4 \text{ and } 0 \leq b \leq 2^{\frac{p-2}{2}} - 1.
\end{align*}
\end{itemize}
\end{cor}

The proofs of these results will be given in section 5. They essentially consist of a comparison of the energy levels $c_{\mu\nu}$ and $\min\{c_\mu, c_\nu\}$, cf. Lemma~\ref{lem_scalarmp} in section~4. 

\begin{rmk}
	If \eqref{eq_assumptions} holds and $p > 8$, 
	we have $2^{\frac{p-2}{2}} - 1 > p-1$ and thus by Theorem~\ref{thm_c_mumu_plus-2},	
	only semitrivial dual ground states occur. 
\end{rmk}

Thanks to a hint by Ev\'{e}quoz (yet unpublished), one can weaken assumptions~\eqref{eq_assumptions} imposing radial symmetry.

\begin{rmk}\label{rmk_restriction}
	If we consider spaces of radial functions and constant coefficients $a, b$, 
	all above Theorems hold under the
	weaker assumption $\frac{2N}{N - 1} < p < 2^\ast$ instead of
	$\frac{2(N+1)}{N-1} < p < 2^\ast$. 
\end{rmk}

Indeed, for the construction of a continuous resolvent
$(- \Delta - \lambda)^{-1}: \: L^{p^\prime}(\R^N, \C) \to L^p(\R^N, \C)$,
\myref refer to a result by Guti\'{e}rrez, Theorem~6 in~\cite{Gutierrez}. A crucial step in its proof, to be found on p. 19 of~\cite{Gutierrez}, is the estimate
\begin{align}\label{eq_Gutierrez}
	\int_{\mathbb{S}^{N-1}} |\mathcal{F}g (r \omega)|^2 \:\mathrm{d}\sigma(\omega)
	\leq
	C \: r^{-\frac{2N}{p'}} \norm{g}^2_{L^p(\R^N)}
	\qquad
	\text{for }
	1 \leq p \leq \frac{2(N+1)}{N+3}
\end{align}
where $g: \R \to \C$ is a Schwartz function and $\mathrm{d}\sigma$ denotes the surface measure on the sphere $\mathbb{S}^{N-1} \subseteq \R^N$. It is a consequence of the Stein-Tomas Theorem, see p. 375 and p. 414 in~\cite{Stein} for $N \geq 3$ and $N = 2$, respectively. In the radial case, however, \myre  is able to show that
$
	\norm{\mathcal{F} h}_{L^\infty_{\mathrm{d}\sigma}(\mathbb{S}^{N-1}, \C)}
	\leq C_p \norm{h}_{L^p(\R^N, \C)} 
$
for all radial Schwartz functions $h: \R \to \C$ and $1 \leq p < \frac{2N}{N+1}$.
Replacing estimate~\eqref{eq_Gutierrez} by
\begin{align*}
	\int_{\mathbb{S}^{N-1}} |\mathcal{F}g (r \omega)|^2 \:\mathrm{d}\sigma(\omega)
	\leq
	\tilde{C} \: r^{-\frac{2N}{p'}} \norm{g}^2_{L^p(\R^N)}
	\qquad
	\text{for }
	1 \leq p < \frac{2N}{N+1}
\end{align*}
for radial Schwartz functions $g: \R \to \C$,
which can be deduced e.g. from Theorem~3.5 in~\cite{Cordoba},
he thus concludes that, in the radial case, the Helmholtz resolvent
$(- \Delta - \lambda)^{-1}: \:  L_{\text{rad}}^{p^\prime}(\R^N, \C) \to L_{\text{rad}}^p(\R^N, \C)$ is continuous for $\frac{2N}{N-1} < p < 2^\ast$.

\absatz

Finally, let us briefly compare our results concerning the occurrence of fully nontrivial ground state solutions to those available in the case of a Schrödinger system, i.e. $\mu, \nu < 0$ in~\eqref{eq_system}. We assume $2 < p < 2^\ast$ and constant coupling $b(x) \equiv \beta \neq 0$; with $\omega := \sqrt{\frac{\nu}{\mu}}$, we discuss
\begin{align*}
	&\begin{cases}
		- \Delta u + \,\,\,\,\: u 
		= \left( |u|^\frac{p}{2} + \beta \, |v|^\frac{p}{2} \right) |u|^{\frac{p}{2} - 2} u		
		& \text{on } \R^N,
		\\
		- \Delta v + \omega^2 v 
		= \left(|v|^\frac{p}{2} + \beta \, |u|^\frac{p}{2} \right) |v|^{\frac{p}{2} - 2} v  		
		& \text{on } \R^N,
		\\
		u, v \in H^1(\R^N).
	\end{cases}
\end{align*}
Sharp characterizations of the occurrence of fully nontrivial ground state solutions
have been provided by the first author in~\cite{Mandel_Schroed} for the cooperative case $\beta > 0$, following pioneering work by Ambrosetti and Colorado~\cite{Ambrosetti}, Maia, Montefusco and Pellacci~\cite{MMP} and others. 
In contrast to the Helmholtz case, the parameter $p$ can be chosen from the full superlinear and subcritical range $2 < p < 2^\ast$ whereas in the Helmholtz case, we use  mapping properties of the resolvent available only for $\frac{2(N+1)}{N-1} < p < 2^\ast$.
Moreover, in order to obtain a suitable dual formulation, our discussion for the Helmholtz system only covers the range $0 \leq \beta \leq p-1$; in particular, we only study cooperative systems. In the Schrödinger case, results for the repulsive case $\beta < 0$ are available as well, see for instance~\cite{Mandel_Schroed_2} and the references therein.
Notice that in the special case $\omega = 1$ and $\mu = \nu$ in~\eqref{eq_system_dual}, the ranges for $p$ and $\beta$ in Theorem~1 and Remark~1(a) of~\cite{Mandel_Schroed} for the Schrödinger case agree with those from Corollary~\ref{cor_plus} above for the Helmholtz case.
Finally, in the situation of Theorem~\ref{thm_c_mumu_plus-1} (b) of the Helmholtz case, the question remains open whether there are threshold values for the existence and non-existence of fully nontrivial ground state solutions such as in Theorem~1 in~\cite{Mandel_Schroed}.

\section{The dual formulation, convexity and the Legendre transform}

In this section, we intend to explain and justify the transition from the nonlinear Helmholtz system~\eqref{eq_system} to \eqref{eq_system_dual}. Let us first note that the system~\eqref{eq_system} can be written in the form
\begin{align*}
	&\begin{cases}
		- \Delta u - \mu u = 
		\partial_s f(x, u, v) 						& \text{on } \R^N,
		\\
		- \Delta v - \nu v \, = 
		\partial_t f(x, u, v)					  		& \text{on } \R^N,
		\\
		u, v \in L^p(\R^N)
	\end{cases}
\end{align*}
where
\begin{align}\label{eq_f}
	f: \R^N \times \R \times \R \to \R, \quad
	f (x, s, t) 
	= \frac{a(x)}{p} \left(|s|^p + 2 b(x) |s|^\frac{p}{2} |t|^\frac{p}{2} + |t|^p\right).
\end{align}
The transformation~\eqref{eq_transform-ubaru} then reads 
$	(\bar{u}(x), \bar{v}(x)) = \nabla_{s, t} f(x, u(x), v(x))$ for $x \in \R^N$.

\absatz

We first provide the definition and the most important properties of the convolution terms in \eqref{eq_system_dual}, referring to section 2 of~\cite{EvequozWeth} and the beginning of section 4 in~\cite{Evequoz_plane} for more details. 
We denote by $\mathcal{F}$ the Fourier transform on the Schwartz space $\mathcal{S}(\R^N, \C)$. 
For $\lambda > 0$ and $\eps > 0$, the operator 
	$- \Delta - (\lambda + \i \eps): \mathcal{S}(\R^N, \C) \to \mathcal{S}(\R^N, \C)$ 
possesses a resolvent given by 
\begin{align*}
	\mathcal{R}_\eps^\lambda: \mathcal{S}(\R^N, \C) \to \mathcal{S}(\R^N, \C), \: \:
	g \mapsto \mathcal{F}^{-1} \left( (|\,\cdot\,|^2 - (\lambda + \i\eps))^{-1} \cdot \mathcal{F}(g) \right).
\end{align*}
By the Limiting Absorption Principle of Guti\'{e}rrez, cf. the introduction before Theorem~6 in~\cite{Gutierrez}, it is possible to pass to the limit $\eps \searrow 0$, which yields an operator 
\begin{align*}
	\mathcal{R}^\lambda: \mathcal{S}(\R^N, \C) \to \mathcal{S}'(\R^N, \C),
	\quad
	\mathcal{R}^\lambda g [w] = \int_{\R^N} (\Phi_\lambda \ast g) \cdot w \: \mathrm{d}x
\end{align*}
where the kernel $\Phi_\lambda: \R^N \to \C$ is given by 
\begin{align}\label{eq_phi}
	\Phi_\lambda (x) = \lambda^\frac{N-2}{2} \: \Phi_1(\sqrt{\lambda} x)
	&=
	 \frac{\i}{4} 
	\left( \frac{\lambda}{4 \pi^2 |x|^2} \right)^\frac{N-2}{4} \: 
	 H^{(1)}_\frac{N-2}{2}(|\sqrt{\lambda} x|)		
	 \quad
	 (x \in \R^N, \lambda > 0),
\end{align}
cf. the text before equation (6) in~\cite{Evequoz_plane} for $N = 2$ and equation (11) in~\cite{EvequozWeth} for $N \geq 3$.
It can be shown that there exists a continuous extension
\begin{equation}\label{eq_HelmholtzResolve}
 	\mathcal{R}^\lambda: L^{p^\prime}(\R^N, \C) \to L^p(\R^N, \C),
	\quad
	\bar{w} \mapsto \Phi_\lambda \ast \bar{w}
	\qquad
	\text{for } \frac{2(N+1)}{N-1} < p < 2^\ast,
\end{equation}
cf. Theorem 2.1 both in~\cite{Evequoz_plane} and in~\cite{EvequozWeth}. Proposition A.1 in~\cite{EvequozWeth} now states that, for 
$f \in L^{p^\prime}(\R^N, \C)$, the function $u := \mathcal{R}^\lambda f \in L^p(\R^N, \C)$ is in fact a strong solution of the linear Helmholtz equation $- \Delta u - \lambda u = f$ on $\R^N$. In this sense, the operator $\mathcal{R}^\lambda$ can be interpreted as a resolvent operator for the linear Helmholtz equation.

\absatz

Being interested in real-valued solutions, we introduce $\Psi_\lambda := \Re \Phi_\lambda$ and, from now on, consider spaces of real-valued functions. Then we have continuity of  $\Psi_{\mu} \ast [ \, \cdot \, ], \Psi_{\nu} \ast [ \, \cdot \, ] : L^{p^\prime}(\R^N) \to L^p(\R^N)$ and we can find solutions of the system~\eqref{eq_system} by solving 
\begin{align}\label{eq_system_mid}
	&\begin{cases}
		 u = \Psi_\mu \ast \partial_s f (x, u, v)		& \text{on } \R^N,
		\\
		v = \Psi_\nu \ast\partial_t f (x, u, v)  		& \text{on } \R^N,
		\\
		u, v \in L^p(\R^N).
	\end{cases}
\end{align}
As mentioned earlier, we aim to reformulate the system~\eqref{eq_system_mid} by, roughly speaking, replacing the functions $u, v \in L^p(\R^N)$ by a corresponding pair 
\begin{align*}
	\bar{u}, \bar{v} \in L^{p^\prime}(\R^N)
	\quad \text{via} \quad
	\bar{u} := \partial_s f (\,\cdot\, , u, v), \: \: \bar{v} := \partial_t f (\,\cdot\, , u, v),
\end{align*}
see also~\eqref{eq_transform-ubaru},
such that the convolutions occur in the linear part of the equations. We will see in Proposition~\ref{prop_fbhb} that, under suitable assumptions on the coefficients $a$ and $b$, this transformation is invertible and preserves the variational structure in the sense that
\begin{align*}
	u = \partial_{\bar{s}} h (\,\cdot\, , \bar{u}, \bar{v}), 
	\quad
	v = \partial_{\bar{t}} h (\,\cdot\, , \bar{u}, \bar{v})
\end{align*} 
with a suitable function $h: \R^N \times \R \times \R \to \R$, which then finally provides a one-to-one correspondence between solutions of the systems \eqref{eq_system_mid} and \eqref{eq_system_dual}. 
It turns out that we have to choose $h(x, \: \cdot \:, \: \cdot \:)$ to be the Legendre transform of $f(x, \: \cdot \:, \: \cdot \:)$ for every fixed $x \in \R^N$.
We remark that, in the case of a single nonlinear Helmholtz equation~\eqref{eq_helmholtz}, $- \Delta w - \lambda w = Q(x) |w|^{p-2} w$ on $\R^N$, the associated change of variables can be done explicitly, $\bar{w}(x) := Q(x) |w(x)|^{p-2} w(x)$ and hence $w(x) = Q(x)^{1 - p^\prime} |\bar{w}(x)|^{p^\prime - 2} \bar{w}(x)$. Notice that the treatment of the term $Q$ slightly differs from that in~\cite{Evequoz_plane, EvequozWeth}.

\absatz

By Theorems 26.5 and 26.6 in~\cite{Convex}, we have the following general result of convex analysis:
\begin{thm}\label{thm_diffb}
	Let $F: \R^2 \to \R$ be differentiable, strictly convex and co-finite.
	Then $\nabla F: \R^2 \to \R^2$ is a homeomorphism,
	and the Legendre transform of $F$,
	\begin{align*}
		H: \R^2 \to \R, \quad
		H(\bar{s}, \bar{t}) := \sup_{(s, t) \in \R^2} \left( s \bar{s} + t \bar{t} - F(s, t)  \right)
	\end{align*}
	is well-defined, differentiable, strictly convex, co-finite and satisfies $\nabla H = (\nabla F)^{-1}$.
\end{thm}
Let us remark that, for a convex function $F: \R^2 \to \R$, co-finiteness is characterized by
\begin{align*}
	\lim_{\lambda \to \infty} \frac{F(\lambda s, \lambda t)}{\lambda} = \infty
	\quad \text{for all } (s, t) \neq (0, 0),
\end{align*} 
cf. the equation before Theorem 26.6 in~\cite{Convex}.
We check that, under the assumptions~\eqref{eq_assumptions} and for fixed $x \in \R^N$, Theorem~\ref{thm_diffb} applies to the function
\begin{align*}
	f(x, \: \cdot \:, \: \cdot \:): \R^2 \to \R, \quad
	f(x, s, t) = \frac{a(x)}{p} \left(|s|^p + 2 b(x) |s|^\frac{p}{2} |t|^\frac{p}{2} + |t|^p\right)
\end{align*}
so that a dual variational formulation for~\eqref{eq_system} is available and given by~\eqref{eq_system_dual}.
\begin{prop}\label{prop_fbhb}
	Let $p > 2$ and $x \in \R^N$ with $a(x) > 0$ and $0 \leq b(x) \leq p - 1$.
	Then the function $f(x, \: \cdot \:, \: \cdot \:)$ 
	is differentiable, strictly convex and co-finite. 
	Hence, its  Legendre transform $h(x, \: \cdot \:, \: \cdot \:)$ is 
	well-defined, differentiable, strictly convex, co-finite and satisfies 
	$\nabla_{\bar{s}, \bar{t}} h(x, \: \cdot \:, \: \cdot \:) = (\nabla_{s, t} f(x, \: \cdot \:, \: \cdot \:))^{-1}$. 	 
\end{prop}
The proof is given in the appendix. Let us emphasize that it is only this existence result which requires the assumption $0 \leq b(x) \leq p-1$ in~\eqref{eq_assumptions}. In some cases, the Legendre transform $h$ can be calculated explicitly. For instance, in the case $b(x) \equiv 1$, we have for all $x \in \R^N$ and $s, t \in \R$
\begin{small}
\begin{align*}
		&f(x, s, t) = \frac{a(x)}{p} \left(|s|^\frac{p}{2} + |t|^\frac{p}{2}\right)^2,
		&& \nabla_{s, t} f (x, s, t) =a(x) \, \left( |s|^\frac{p}{2} + |t|^\frac{p}{2} \right) 
		\svec{|s|^{\frac{p}{2}-2} s}{|t|^{\frac{p}{2}-2} t},
		\\
		&h (x, \bar{s}, \bar{t}) = \frac{a(x)^{1 - p^\prime}}{p^\prime} 
		\left(|\bar{s}|^\frac{p}{p-2} + |\bar{t}|^\frac{p}{p-2}  \right)^\frac{p-2}{p-1},
		&& \nabla_{\bar{s}, \bar{t}} h(x, \bar{s}, \bar{t}) 
		= \left( a(x) \left(|\bar{s}|^\frac{p}{p-2} + |\bar{t}|^\frac{p}{p-2} \right)  \right)^{1 - p^\prime} 
		\svec{|\bar{s}|^\frac{4-p}{p-2} \bar{s}}{|\bar{t}|^\frac{4-p}{p-2} \bar{t}}.
\end{align*}
\end{small}

In the following sections we will need some properties of $h(x, \: \cdot \:, \: \cdot \:)$ which are listed next and will be proved in the appendix.
\begin{lem}\label{lem_prop_hb}
Let $p > 2$ and $x \in \R^N$ with $a(x) > 0$, $0 \leq b(x) \leq p - 1$. Then, for $\bar{s}, \bar{t} \in \R$,  
\begin{enumerate}[label={(\alph*)}]
\item	$	h(x, \bar{s}, \bar{t}) = \dfrac{a(x)^{1 - p^\prime}}{p^\prime} \left[\sup\limits_{\sigma > 0} 
		\dfrac{|\bar{s}| + \sigma |\bar{t}|}{ (1 + 2b(x) \, \sigma^\frac{p}{2} + \sigma^p)^\frac{1}{p}}
		\right]^{p^\prime}$,
\item $h(x, \bar{s}, \bar{t}) = h(x, \bar{t}, \bar{s}) = h(x, -\bar{s}, \bar{t})$,
\item $h(x, \bar{s}, \bar{t}) = 
			\dfrac{1}{p^\prime} \nabla_{\bar{s}, \bar{t}} h(x, \bar{s}, \bar{t}) \cdot \svec{\bar{s}}{\bar{t}}$,
\item $h(x, \bar{s}, 0) = \dfrac{a(x)^{1 - p^\prime}}{p^\prime} |\bar{s}|^{p^\prime}$ as well as 
			$h(x, \bar{s}, \bar{s}) = \dfrac{2 a(x)^{1 - p^\prime}}{p^\prime} 
			 (1+b(x))^{1 - p^\prime} |\bar{s}|^{p^\prime} $,
\item $\dfrac{1}{p^\prime} (a(x) (1 + b(x)))^{1 - p^\prime} 
	\big( |\bar{s}|^{p^\prime} + |\bar{t}|^{p^\prime} \big)
	\leq h(x, \bar{s}, \bar{t}) \leq \dfrac{1}{p^\prime} a(x)^{1-p^\prime}
	\big( |\bar{s}|^{p^\prime} + |\bar{t}|^{p^\prime} \big)$.
\end{enumerate}
\end{lem}

If we additionally impose that the coefficients $a, b: \R^N \to \R$ are measurable, we conclude that for $\bar{u}, \bar{v} \in \L{p^\prime}$, the mapping 
\begin{align*}
	x \mapsto h(x, \bar{u}(x), \bar{v}(x)) &= \sup_{s, t \in \R} \left( s\bar{u}(x) + t\bar{v}(x) - f(x, s, t) \right) 
	\\
	&= \sup_{s, t \in \Q} \left( s\bar{u}(x) + t\bar{v}(x) - f(x, s, t) \right)
\end{align*}
is measurable since it is a pointwise supremum of countably many measurable functions. Moreover, when combined with (e) of the previous Lemma, we have $h(\,\cdot\,, \bar{u}, \bar{v}) \in \L{1}$ and the functional $J_{\mu\nu}$ as introduced in equation~\eqref{eq_Jmunub} is well-defined and continuously Fr\'{e}chet differentiable.
In particular, for $\bar{u}, \bar{v} \in \L{p^\prime}$, property (c) yields the identity
\begin{align}\label{eq_JprimeJ}
	J_{\mu\nu}'(\bar{u}, \bar{v})[\bar{u}, \bar{v}]
	= 
	\int_{\R^N} p^\prime \, h(x, \bar{u}, \bar{v}) \: \mathrm{d}x
		- \int_{\R^N} \bar{u} \Psi_\mu \ast \bar{u} + \bar{v} \Psi_\nu \ast \bar{v} \: \mathrm{d}x.
\end{align}

\section{Proof of Theorem~\ref{thm_existence}}

In this section, we give the proof of Theorem~\ref{thm_existence}. This will be achieved using the Mountain Pass Theorem and following the ideas in~\cite{EvequozWeth}. We endow the product space $\LL{p^\prime}$ with the norm given by $\norm{(\bar{u}, \bar{v})}_{p^\prime} := \norm{\bar{u}}_{p^\prime} + \norm{\bar{v}}_{p^\prime}$ for $\bar{u}, \bar{v} \in \L{p^\prime}$ and collect the following two major auxiliary results:

\begin{lem}[Mountain Pass Geometry, see Lemma 4.2 in \cite{EvequozWeth}]\label{lem_MPgeo}
	Assuming~\eqref{eq_assumptions}, the functional $J_{\mu\nu}$ has the following 
	properties:
	\begin{itemize}
	\item[(i)]
	There exist $\delta > 0$ and $\rho \in (0, 1)$ 
	with the property that, for $(\bar{u}, \bar{v}) \in \LL{p^\prime}$,
	$J_{\mu\nu}(\bar{u}, \bar{v})  > 0$
	if $0 < \norm{(\bar{u}, \bar{v})}_{p^\prime} \leq \rho$
	and $J_{\mu\nu}(\bar{u}, \bar{v}) \geq \delta$ 
	if $\norm{(\bar{u}, \bar{v})}_{p^\prime} = \rho$.
	\item[(ii)]
	There exists $(\bar{u}_1, \bar{v}_1) \in \LL{p^\prime}$ with $\norm{(\bar{u}_1, \bar{v}_1)}_{p^\prime} > 1$ and
	$J_{\mu\nu}(\bar{u}_1, \bar{v}_1) < 0$.
	\item[(iii)]
	Every Palais-Smale sequence for $J_{\mu\nu}$ is bounded in $\LL{p^\prime}$.
	\end{itemize}
\end{lem}
\begin{lem}[Existence of Palais-Smale sequence, see Lemma 6.1 in \cite{EvequozWeth}]\label{lem_PSseq}
	Assuming~\eqref{eq_assumptions}, 
	there exists a bounded Palais-Smale sequence $(\bar{u}_n, \bar{v}_n)_{n \in \N}$ in $\LL{p^\prime}$
	for $J_{\mu\nu}$ at the level $c_{\mu\nu}$ as given in equation~\eqref{eq_mplevel}.
\end{lem}
Up to minor modifications, both results can be proved in the same way as the corresponding scalar results, Lemma~4.2 and Lemma~6.1 in~\cite{EvequozWeth} for $N \geq 3$, which also hold for $N = 2$ as explained at the beginning of the proof of Theorem~1.3 (b) of~\cite{Evequoz_plane}. We thus omit the proof.

\begin{proof}[\textbf{Proof of Theorem~\ref{thm_existence}}]
	This proof mainly follows the lines of the proof of Theorem 1.3(b) in~\cite{Evequoz_plane} for $N = 2$ and Theorem 6.2 in~\cite{EvequozWeth} for $N \geq 3$, respectively, which we will refer to as the scalar case. We will therefore focus on those parts which differ due to the fact that we discuss a system of equations.

	Let $(\bar{u}_n, \bar{v}_n)_{n \in \N}$ denote a bounded Palais-Smale sequence 
	at the level $c_{\mu\nu}$  which exists by Lemma~\ref{lem_PSseq}; 
	then w.l.o.g. $\bar{u}_n \rightharpoonup \bar{u}$ 
	and $\bar{v}_n \rightharpoonup \bar{v}$ as $n \to \infty$ weakly in $\L{p^\prime}$. 
	We perform a concentration compactness argument which relies on the periodicity of the coefficients $a, b$.
	
	\step{1}{(Nonvanishing.) 
	There exists a ball $B \subseteq \R^N$ such that, up to a subsequence
	and up to translations, $\inf_{n \in \N} \: \int_{B} h(x, \bar{u}_n, \bar{v}_n) \: \mathrm{d}x > 0$.}
	
	As in the scalar case, definition~\eqref{eq_Jmunub} and identity~\eqref{eq_JprimeJ} imply, as $n \to \infty$, 
	\begin{align*}
		\int_{\R^N}  \bar{u}_n \Psi_\mu \ast \bar{u}_n + \bar{v}_n \Psi_\nu \ast \bar{v}_n \: \mathrm{d}x
		&= \frac{2p}{p-2}
		\left[  J_{\mu\nu}(\bar{u}_n, \bar{v}_n) 	
		- \frac{1}{p^\prime} J_{\mu\nu}^\prime(\bar{u}_n, \bar{v}_n)[\bar{u}_n, \bar{v}_n] \right]
		\to \frac{2p}{p-2} \cdot c_{\mu\nu}.
	\end{align*}
	As $c_{\mu\nu} > 0$ due to \eqref{eq_mplevel} and Lemma~\ref{lem_MPgeo} (i), 
	we conclude 
	$\limsup_{n\to\infty} \int_{\R^N}  \bar{u}_n \Psi_\mu \ast \bar{u}_n \: \mathrm{d}x > 0$ or 
	$\limsup_{n\to\infty} \int_{\R^N}  \bar{v}_n \Psi_\nu \ast \bar{v}_n \: \mathrm{d}x > 0$.
	We apply the (scalar) Nonvanishing Theorems, Theorem~3.1 in~\cite{Evequoz_plane} for $N = 2$
	and Theorem~3.1 in~\cite{EvequozWeth} for $N \geq 3$ to find $R, \zeta > 0$ and 
	$(x_n)_{n \in \N} \subseteq \R^N$ such that, up to a subsequence,
	$\int_{B_R(x_n)} |\bar{u}_n|^{p^\prime}+ |\bar{v}_n|^{p^\prime} \: \mathrm{d}x \geq \zeta$
	holds for all $n \in \N$. 
	Possibly enlarging the radius $R$, we may w.l.o.g. assume $x_n \in \Z^N$ for all $n \in \N$.	
	By Lemma~\ref{lem_prop_hb} (e), and with 
	$\gamma_p := \frac{1}{p'} (a_+ (1 + b_+))^{1 - p^\prime}$, we have
	\begin{align}\label{eq_nonvanish_first}
		\int_{B_R(x_n)} h(x, \bar{u}_n, \bar{v}_n) \: \mathrm{d}x 
		\geq \gamma_p \, \zeta
		\quad \text{for all } n \in \N.
	\end{align}
	Next, we introduce the shifted 
	functions	$\bar{U}_n(x) := \bar{u}(x_n + x)$ and $\bar{V}_n(x) := \bar{v}(x_n + x)$ 
	for $n \in \N, x \in \R^N$. 
	We note that, due to the periodicity of the coefficients $a, b$ and since $x_n \in \Z^N$, 
	the Legendre transform is invariant under such translations in the sense that 
	\begin{align*}
		h(x, \bar{U}_n(x), \bar{V}_n(x))
		&= \sup_{s, t \in \R} \left( s \bar{U}_n(x) + t \bar{V}_n(x) - f(x, s, t) \right)
		\\
		&= \sup_{s, t \in \R} \left( s \bar{U}_n(x) + t \bar{V}_n(x) - f(x_n + x, s, t) \right)
		\\
		&= h(x_n + x, \bar{U}_n(x), \bar{V}_n(x)) 
		\\
		&= h(x_n + x, \bar{u}_n(x_n + x), \bar{v}_n(x_n + x))
	\end{align*}
	for almost all $x \in \R^N$ and every $n \in \N$. Thus, and due to~\eqref{eq_nonvanish_first},
	\begin{align}\label{eq_nonvanish}
		\int_{\R^N} h(x, \bar{U}_n, \bar{V}_n) \: \mathrm{d}x
		= \int_{\R^N} h(x, \bar{u}_n, \bar{v}_n) \: \mathrm{d}x
		\quad \mathrm{and} \quad
		\int_{B_R(0)} h(x, \bar{U}_n, \bar{V}_n) \: \mathrm{d}x \geq \gamma_p \, \zeta.
	\end{align}
	With that, arguing as in the scalar case, we obtain that
	$(\bar{U}_n, \bar{V}_n)_{n \in \N}$ is a bounded Palais-Smale sequence for $J_{\mu\nu}$ 
	to the level $c_{\mu\nu}$. Hence, w.l.o.g., there exist $\bar{U}, \bar{V} \in \L{p^\prime}$ 
	with $\bar{U}_n \rightharpoonup \bar{U}$ 
	and $\bar{V}_n \rightharpoonup \bar{V}$ as $n \to \infty$ weakly in $\L{p^\prime}$. 
	\vspace*{0.5cm}\\
	We intend to prove that
	\begin{align*}
		\int_{B_R(0)} h(x, \bar{U}_n, \bar{V}_n) \: \mathrm{d}x 
		\to
		\int_{B_R(0)} h(x, \bar{U}, \bar{V}) \: \mathrm{d}x
		\qquad
		\text{as } n \to \infty
	\end{align*}
	and that hence, due to the inequality in~\eqref{eq_nonvanish}, $(\bar{U}, \bar{V}) \neq (0, 0)$. 
	To this end, we need the following auxiliary result:
	
	\step{2}{We have
	$\mathds{1}_{B_R(0)} \cdot \nabla_{\bar{s}, \bar{t}} h(\:\cdot\:, \bar{U}_n, \bar{V}_n)
	\to
	\mathds{1}_{B_R(0)} \cdot \nabla_{\bar{s}, \bar{t}} h(\:\cdot\:, \bar{U}, \bar{V})$ 
	strongly in $\LL{p}$ as $n \to \infty$.}
	
	Let $\varphi, \psi \in \L{p^\prime}$ and $\tilde{\varphi} := \varphi \cdot \mathds{1}_{B_R(0)}, 
	\tilde{\psi} := \psi \cdot \mathds{1}_{B_R(0)}$. We estimate for $m, n \in \N$
	\begin{align*}
		& \left| 
		\int_{\R^N} \left( \mathds{1}_{B_R(0)} \cdot 
		\nabla_{\bar{s}, \bar{t}} h(\:\cdot\:, \bar{U}_n, \bar{V}_n)
		- \mathds{1}_{B_R(0)} \cdot 
		\nabla_{\bar{s}, \bar{t}} h(\:\cdot\:, \bar{U}_m, \bar{V}_m) \right) \cdot 
		\svec{\varphi}{\psi} \: \mathrm{d}x
		\right|
		\\
		& \quad = \left|
		J_{\mu\nu}^\prime(\bar{U}_n, \bar{V}_n)[\tilde{\varphi}, \tilde{\psi}] 
		- J_{\mu\nu}^\prime(\bar{U}_m, \bar{V}_m)[\tilde{\varphi}, \tilde{\psi}]
		+ \int_{\R^N} \tilde{\varphi} \Psi_\mu \ast (\bar{U}_n - \bar{U}_m)
		+ \tilde{\psi} \Psi_\nu \ast (\bar{V}_n - \bar{V}_m) \: \mathrm{d}x
		\right|
		\\
		& \quad \leq 
		C_{nm} \norm{(\varphi, \psi)}_{p^\prime},
	\end{align*}
	where
	\begin{align*}
		C_{nm} = &\norm{J_{\mu\nu}^\prime(\bar{U}_n, \bar{V}_n)}_{\mathcal{L}(\LL{p^\prime}, \R)}
		+ \norm{J_{\mu\nu}^\prime(\bar{U}_m, \bar{V}_m)}_{\mathcal{L}(\LL{p^\prime}, \R)} 
		\\
		&+ \norm{\mathds{1}_{B_R(0)} \cdot \Psi_\mu \ast (\bar{U}_n - \bar{U}_m)}_p
		+ \norm{\mathds{1}_{B_R(0)} \cdot \Psi_\nu \ast (\bar{V}_n - \bar{V}_m)}_p.
	\end{align*} 
	Then, we have $C_{nm} \to 0$ as $m, n \to \infty$
	since $\norm{J_{\mu\nu}^\prime(\bar{U}_n, \bar{V}_n)}_{\mathcal{L}(\LL{p^\prime}, \R)} \to 0$ by the
	Palais-Smale property and since the operator
	\begin{align*}
		\L{p'} \to \L{p}, 
		\quad
		g \mapsto \mathds{1}_{B_R(0)} \cdot \Psi_\mu \ast g
	\end{align*}		
	is	compact, cf. Lemma~4.1 in~\cite{EvequozWeth} for $N \geq 3$ and the corresponding result 
	at the beginning of section 3 of~\cite{Evequoz_plane} for $N = 2$.

	\absatz	
	
	By duality, 
	$(\mathds{1}_{B_R(0)} \cdot \nabla_{\bar{s}, \bar{t}} h(\:\cdot\:, \bar{U}_n, \bar{V}_n))_{n \in \N}$ 
	is a Cauchy	sequence in $\LL{p}$. We thus find $U, V \in L^p(B_R(0))$ with
	$\nabla_{\bar{s}, \bar{t}} h(\:\cdot\:, \bar{U}_n, \bar{V}_n) \to (U, V)$ as $n \to \infty$ 
	in $L^p(B_R(0)) \times L^p(B_R(0))$ and, up to a subsequence, 
	pointwise almost everywhere on $B_R(0)$. 
	
	As $\nabla_{s, t} f(x, \:\cdot\:, \:\cdot\: )$ is a homeomorphism on $\R^2$ for all $x \in \R^N$, 
	we have $(\bar{U}_n, \bar{V}_n) \to \nabla_{s, t} f(\:\cdot\:, U, V)$ 
	almost everywhere 	on $B_R(0)$ as $n \to \infty$. 
	Since the sequences $(\bar{U}_n)_{n\in\N}, (\bar{V}_n)_{n\in\N}$ 
	are bounded in $L^{p^\prime}(B_R(0))$, 
	Theorem 1 in~\cite{Jakszto} implies that 
	$(\bar{U}_n, \bar{V}_n) \rightharpoonup \nabla_{s, t} f(\:\cdot\:, U, V)$ 
	weakly in $L^{p^\prime}(B_R(0)) \times L^{p^\prime}(B_R(0))$ as $n \to \infty$.
	However, from the end of Step 1, we know that 
	$(\bar{U}_n, \bar{V}_n) \rightharpoonup  (\bar{U}, \bar{V})$ 
	weakly in $\LL{p^\prime}$	as $n \to \infty$. Uniqueness of the weak limit now implies 
	$(\bar{U}, \bar{V}) = \nabla_{s, t} f(\:\cdot\:, U, V)$ in $L^{p'}(B_R(0)) \times L^{p'}(B_R(0))$,
	hence
	$(U, V) = \nabla_{\bar{s}, \bar{t}} h(\:\cdot\:, \bar{U}, \bar{V})\big|_{B_R(0)}$
	and
	\begin{align*}
		\mathds{1}_{B_R(0)} \cdot \nabla_{\bar{s}, \bar{t}} h(\:\cdot\:, \bar{U}_n, \bar{V}_n) 
		\to 
		\mathds{1}_{B_R(0)} \cdot \nabla_{\bar{s}, \bar{t}} h(\:\cdot\:, \bar{U}, \bar{V})
		\quad
		\text{as } n \to \infty \text{ in } \LL{p}.
	\end{align*}
	 
	\step{3}{Conclusion.}	
	
	We find with Lemma~\ref{lem_prop_hb} (c)
	\begin{align*}
		&\left|
		\int_{B_R(0)} h(x, \bar{U}_n, \bar{V}_n) - h(x, \bar{U}, \bar{V}) \: \mathrm{d}x
		\right|
		\\
		& \quad = 
		\frac{1}{p^\prime} \left|
		\int_{B_R(0)} \nabla_{\bar{s}, \bar{t}} h(x, \bar{U}_n, \bar{V}_n) 
		\cdot \svec{\bar{U}_n}{\bar{V}_n} 
		- \nabla_{\bar{s}, \bar{t}} h(x, \bar{U}, \bar{V}) 
		\cdot \svec{\bar{U}}{\bar{V}} \: \mathrm{d}x
		\right|
			\\
		& \quad \leq \frac{1}{p^\prime} 
		\norm{ \mathds{1}_{B_R(0)} \left(\nabla_{\bar{s}, \bar{t}} h(\:\cdot\:, \bar{U}_n, \bar{V}_n) 
		- \nabla_{\bar{s}, \bar{t}} h(\:\cdot\:, \bar{U}, \bar{V}) \right)}_p 
		\norm{(\bar{U}_n, \bar{V}_n)}_{p^\prime}
		\\
		& \quad \quad 
		+\frac{1}{p^\prime} \left|
		\int_{B_R(0)} \nabla_{\bar{s}, \bar{t}} h(x, \bar{U}, \bar{V}) \cdot 
		\left[ \svec{\bar{U}_n}{\bar{V}_n}	 - \svec{\bar{U}}{\bar{V}} \right] \: \mathrm{d}x
		\right| 
	\end{align*}
	and both terms tend to zero by Step 2 and Step 1, respectively. 
	Hence, in view of the inequality in~\eqref{eq_nonvanish}, we have
	\begin{align*}
			\int_{B_R(0)} h(x, \bar{U}, \bar{V}) \: \mathrm{d}x
			= \lim_{n \to \infty} \int_{B_R(0)} h(x, \bar{U}_n, \bar{V}_n) \: \mathrm{d}x
			\geq \gamma_p \, \zeta > 0,
	\end{align*}	 
	which shows (via (e) of Lemma~\ref{lem_prop_hb}) that the weak limit satisfies 
	$(\bar{U}, \bar{V}) \neq (0, 0)$.
	
	What remains to prove is that indeed $J_{\mu\nu}(\bar{U}, \bar{V}) = c_{\mu\nu}$ and 
	$J_{\mu\nu}^\prime(\bar{U}, \bar{V}) = 0$.
	As in the scalar case, this is a consequence of the fact that 
	$(\bar{U}_n, \bar{V}_n)_{n\in\N}$ is a Palais-Smale sequence 
	which converges weakly to $(\bar{U}, \bar{V})$; for details cf. the last lines of the proof
	of Theorem~6.2 in~\cite{EvequozWeth} and of Theorem~1.3~(b) in~\cite{Evequoz_plane}, respectively.

	\absatz	
	
	Finally, letting $u := \partial_{\bar{s}} h(\,\cdot\, , \bar{u}, \bar{v})$ and 
	$v := \partial_{\bar{t}} h(\,\cdot\, , \bar{u}, \bar{v})$, it can be shown as in 
	Lemma~4.3 in~\cite{EvequozWeth} that this provides
	a strong solution of~\eqref{eq_system} and that 
	$u, v \in W^{2,q}(\R^N) \cap C^{1, \alpha}(\R^N)$ 
	for all $p \leq q < \infty, 0 < \alpha < 1$.
\end{proof}

\section{Energy levels and an inf-sup characterization of a dual ground state}

As announced earlier, the proofs of Theorems~\ref{thm_c_mumu_plus-1}~and~\ref{thm_c_mumu_plus-2} essentially consist of a comparison of energy levels. Thus, the following alternative characterization of the mountain pass level is a crucial ingredient. In preparation, we define 
$F_{\mu\nu}: 	\LL{p^\prime} \to (0, \infty]$ by $F_{\mu\nu}(0, 0) := \infty$ and, for 
$(\bar{u}, \bar{v}) \in \LL{p^\prime} \setminus \{(0, 0)\}$,
\begin{equation}\label{eq_Fmunu}
	 F_{\mu\nu}(\bar{u}, \bar{v}) := 
	\frac{p-2}{2p} \left(
	\frac{\left[\int_{\R^N} p^\prime \, h(x, \bar{u}, \bar{v}) \: \mathrm{d}x\right]^\frac{1}{p^\prime}}
	{\left[\int_{\R^N} \bar{u} \Psi_\mu \ast \bar{u} 
	+ \bar{v} \Psi_\nu \ast \bar{v} \: \mathrm{d}x\right]_+^\frac{1}{2}}
	\right)^\frac{2p}{p-2}.
\end{equation} 
Notice that the denominator may vanish also for $(\bar{u}, \bar{v}) \neq (0, 0)$ due to the oscillatory nature of the integral kernels; in this case, the definition yields $F_{\mu\nu}(\bar{u}, \bar{v}) = + \infty$.
With definition~\eqref{eq_Jmunub} and Lemma~\ref{lem_prop_hb} (a), we have
\begin{align*}
	J_{\mu\nu} (\tau\bar{u}, \tau\bar{v}) 
	= \frac{\tau^{p^\prime}}{p^\prime} \int_{\R^N} p^\prime \, h(x, \bar{u}, \bar{v}) \: \mathrm{d}x 
	- \frac{\tau^2}{2} \int_{\R^N} \bar{u} \Psi_\mu \ast \bar{u} 
	+ \bar{v} \Psi_\nu \ast \bar{v} \: \mathrm{d}x 
\end{align*}
for $\tau > 0$.
The mapping $\tau \mapsto J_{\mu\nu} (\tau\bar{u}, \tau\bar{v})$ possesses a critical point on $(0, \infty)$ if and only if 
$\int_{\R^N} \bar{u} \Psi_\mu \ast \bar{u} + \bar{v} \Psi_\nu \ast \bar{v} \: \mathrm{d}x > 0$; in this case, 
the critical point is unique and a global maximum.
A straightforward calculation shows
\begin{equation*}
	\sup_{\tau > 0} J_{\mu\nu}(\tau\bar{u}, \tau\bar{v}) = F_{\mu\nu}(\bar{u}, \bar{v}).
\end{equation*}
As a result of one-dimensional calculus, we obtain the following Lemma which provides an inf-sup characterization of the mountain pass level $c_{\mu\nu}$ defined in equation~\eqref{eq_mplevel}. We do not present its proof; variants of it can be found in the literature, e.g. the first lines of the proof of Proposition~2.1 in~\cite{Mandel_Saturation} and the text before Lemma~2.1 in~\cite{Evequoz_high}. 

\begin{lem}\label{lem_mplevel}
Under the assumptions given in~\eqref{eq_assumptions}, the mountain pass level $c_{\mu\nu}$ as defined in equation~\eqref{eq_mplevel} can be characterized as follows:
\begin{align*}
	c_{\mu\nu} 
	&= 
	\inf \brac{
	F_{\mu\nu}(\bar{u}, \bar{v}): \: \: 
	(\bar{u}, \bar{v}) \in \LL{p^\prime}	}.
\end{align*}
Moreover, $(\bar{u}_0, \bar{v}_0) \in \LL{p^\prime}$ is a minimizer of the functional $F_{\mu\nu}$ if and only if it is a nonzero multiple of a critical point of $J_{\mu\nu}$ on the mountain pass level $c_{\mu\nu}$.
\end{lem}

These results also apply in the scalar case discussed in Remark~\ref{rmk_scalar}; we define the functional
\begin{align}\label{eq_Emu}
	E_{\mu}: 
	\L{p^\prime} \to (0, \infty],
	\quad
	 E_{\mu}(\bar{u}) := 
	\frac{p-2}{2p} \left(
	\frac{\left[ \int_{\R^N} a(x)^{1 - p^\prime} |\bar{u}|^{p^\prime} \: \mathrm{d}x \right]^\frac{1}{p^\prime}}
	{\left[ \int_{\R^N} \bar{u} \Psi_\mu \ast \bar{u} \: \mathrm{d}x \right]_+^\frac{1}{2}}
	\right)^\frac{2p}{p-2}
\end{align}
again with $E_\mu (0) := + \infty$. Then $E_\mu(\bar{u}) = F_{\mu\nu}(\bar{u}, 0)$ 
for $\bar{u} \in \L{p'}$ by Lemma~\ref{lem_prop_hb} (d)
and the scalar mountain pass level $c_\mu$ is the infimum of $E_\mu$, attained in particular at critical points of $I_\mu$ on the level $c_\mu$.
We derive some direct consequences describing the relation between the mountain pass level associated with the system~\eqref{eq_system_dual} and the scalar mountain pass level. Recall that a critical point of $J_{\mu\nu}$ on the mountain pass level $c_{\mu\nu}$ is said to be a dual ground state.
\begin{lem}\label{lem_scalarmp}
We assume that conditions~\eqref{eq_assumptions} hold. Then we have the following:
\begin{itemize}
\item[(i)] The inequality $c_{\mu\nu} \leq c_\mu$ holds.
\item[(ii)] If $(\bar{u}_0, 0)$ is a semitrivial dual ground state of $J_{\mu\nu}$,
then $c_{\mu\nu} = c_\mu$ and $\bar{u}_0$ is a dual ground state of the scalar functional $I_\mu$.
\end{itemize}
\end{lem}

\begin{proof}
	\begin{itemize} 
\item[(i)]
	This is a consequence of the fact that 
	$I_\mu(\bar{u}) = J_{\mu\nu}(\bar{u}, 0)$ for $\bar{u} \in \L{p^\prime}$.

\item[(ii)]
	Assume that $(\bar{u}_0, 0) \in \LL{p'}$ is a dual ground state 
	of $J_{\mu\nu}$, i.e. $J'_{\mu\nu}(\bar{u}_0, 0) = 0$ and  $J_{\mu\nu}(\bar{u}_0, 0) = c_{\mu\nu}$.
	As we have $I_\mu(\bar{w}) = J_{\mu\nu}(\bar{w}, 0)$ for all $w \in \L{p'}$, this implies
	$I_\mu^\prime(\bar{u}_0) = 0$ and $I_\mu(\bar{u}_0) = c_{\mu\nu}$. 
	Then, with~\eqref{eq_Emu} and Lemma~\ref{lem_mplevel},
	\begin{align*}
		c_\mu = \inf_{\bar{u} \in \L{p'}}  E_\mu (\bar{u}) 
		\leq E_\mu (\bar{u}_0) = F_{\mu\nu} (\bar{u}_0, 0) = c_{\mu\nu} \overset{(i)}{\leq} c_\mu,	
	\end{align*}		
	we conclude $c_{\mu\nu} = c_\mu$ and therefore $\bar{u}_0$ is a dual ground state of $I_\mu$. 

\end{itemize}
\end{proof}

\section{Proof of Theorems~\ref{thm_c_mumu_plus-1}~and~\ref{thm_c_mumu_plus-2}}

By  $h_\pm: \R^2 \to \R$ we denote the Legendre transforms of the functions 
	$f_\pm: \R^2 \to \R, \: \: 
	f_\pm (s, t) := \frac{1}{p} \left( |s|^p + 2b_\pm \: |s|^\frac{p}{2} |t|^\frac{p}{2} + |t|^p \right)$.	
We recall that $a_- \leq a(x) \leq a_+$ and $b_- \leq b(x) \leq b_+$ hold for almost all $x \in \R^N$.
As a direct consequence of Lemma~\ref{lem_prop_hb} (a), we have the following chain of  inequalities for all $\bar{s}, \bar{t} \in \R$ and $x \in \R^N$:
\begin{equation}\label{eq_mono-a}
	a_+^{1 - p^\prime} h_+(\bar{s}, \bar{t})
	\leq a(x)^{1 - p^\prime} h_+(\bar{s}, \bar{t})
	\leq h(x, \bar{s}, \bar{t})
	\leq a(x)^{1 - p^\prime} h_-(\bar{s}, \bar{t})
	\leq a_-^{1 - p^\prime} h_-(\bar{s}, \bar{t}).
\end{equation}
We recall that, as in Lemma~\ref{lem_prop_hb} (a), we have for $\bar{s}, \bar{t} \in \R$
\begin{equation}\label{eq_hb}
	h_\pm(\bar{s}, \bar{t}) = \frac{1}{p^\prime} \left[\sup_{\sigma > 0} 
		\frac{|\bar{s}| + \sigma |\bar{t}|}{ (1 + 2b_\pm \, \sigma^\frac{p}{2} + \sigma^p)^\frac{1}{p}}
		\right]^{p^\prime}.
\end{equation}

\begin{proof}[\textbf{Proof of Theorem~\ref{thm_c_mumu_plus-1} (a)}]
	We consider a minimizer $\bar{w} \in \L{p^\prime}$, $\bar{w} \neq 0$ of the scalar functional $E_\mu$. 
	Multiplying with a suitable constant, we assume w.l.o.g. 
	$\int_{\R^N} \bar{w} \Psi_\mu \ast \bar{w} \: \mathrm{d}x = 1$ and hence, 
	recalling definition~\eqref{eq_Emu},
	\begin{align}\label{eq_prf_thm2a}
		c_\mu 
		= E_\mu (\bar{w}) 
		= \frac{p-2}{2p} \left( \int_{\R^N} a(x)^{1 - p^\prime} |\bar{w}|^{p^\prime} 
		\: \mathrm{d}x \right)^{\frac{2(p-1)}{p-2}} 
		= \frac{p-2}{2p} \norm{a^{-1/p} \bar{w}}_{p'}^{\frac{2p}{p-2}}.
	\end{align}
		
	Theorem 1.1 in \cite{Evequoz_plane} for $N = 2$ and Lemma 4.3 in \cite{EvequozWeth} for $N \geq 3$, 
	respectively, assure that $\bar{w}$ is continuous. Moreover, equations (6) in 
	\cite{Evequoz_plane} and (11), (12) in \cite{EvequozWeth} imply that $\Psi_\nu > 0$ near zero. 
	Hence there exist $r > 0$ and $x_0 \in \R^N$ with $\bar{w} > 0$ on $B_r(x_0)$ (or with
	$\bar{w} < 0$ on $B_r(x_0)$) and $\Psi_\nu > 0$ on $B_{2r}(0)$.  We then have 
	\begin{align*}
		q := 
		\int_{\R^N} \left( \bar{w} \mathds{1}_{B_r(x_0)} \right) 
		\Psi_\nu \ast \left( \bar{w} \mathds{1}_{B_r(x_0)} \right)
		\: \mathrm{d}x
		= \int_{B_r(x_0)} \int_{B_r(x_0)}  \bar{w}(y) \bar{w}(z) \Psi_\nu(y-z) \: \mathrm{d}y \mathrm{d}z > 0
	\end{align*}
	and estimate for sufficiently small $\eta > 0$:
	\begin{align*}
		c_{\mu\nu}
		&\leq 
		F_{\mu\nu} (\bar{w}, \eta \, \bar{w} \mathds{1}_{B_r(x_0)})
		\\
		&=
		\frac{p-2}{2p}  \left(
		\frac{\left( \int_{\R^N} p^\prime \, h(x, \bar{w}, \eta \, \bar{w} \mathds{1}_{B_r(x_0)}) 	
		\: \mathrm{d}x \right)^\frac{1}{p^\prime}}
		{\left(\int_{\R^N} \bar{w} \Psi_\mu \ast \bar{w} 
		+ \eta^2 (\bar{w} \mathds{1}_{B_r(x_0)}) \Psi_\nu \ast (\bar{w} \mathds{1}_{B_r(x_0)}) 
		\: \mathrm{d}x \right)_+^\frac{1}{2}}	
		\right)^\frac{2p}{p-2} 
		\\
		&=
		\frac{p-2}{2p}  \left(
		\frac{\left( \int_{\R^N} p^\prime \, h(x, \bar{w}, \eta \, \bar{w} \mathds{1}_{B_r(x_0)}) 	
		\: \mathrm{d}x \right)^\frac{1}{p^\prime}}
		{(1	+ \eta^2 q)^\frac{1}{2}}	
		\right)^\frac{2p}{p-2} 
		\\
		&\overset{\eqref{eq_mono-a}}{\leq}
		\frac{p-2}{2p}  \left(
		\frac{\left( \int_{\R^N} p^\prime \, a(x)^{1 - p^\prime} h_-(\bar{w}, \eta \, \bar{w} \mathds{1}_{B_r(x_0)}) 	
		\: \mathrm{d}x \right)^\frac{1}{p^\prime}}
		{(1	+ \eta^2 q)^\frac{1}{2}}	
		\right)^\frac{2p}{p-2} 
		\\
		&\overset{\eqref{eq_hb}}{=}
		\frac{p-2}{2p}  \left(
		\frac{1}{{(1	+ \eta^2 q)^\frac{1}{2}}}
		\left( \int_{\R^N} a(x)^{1 - p^\prime} \,
		\left(\sup_{\sigma > 0} \frac{|\bar{w}(x)| + \sigma \eta |\bar{w}(x)|\mathds{1}_{B_r(x_0)}(x) }
		{(1 + 2 b_- \sigma^{\frac{p}{2}} + \sigma^p)^{\frac{1}{p}}}\right)^{p^\prime}
		\: \mathrm{d}x \right)^\frac{1}{p^\prime}
		\right)^\frac{2p}{p-2} 
		\\
		&\leq
		\frac{p-2}{2p} 
		\norm{a^{-1/p} \bar{w}}_{p'}^{\frac{2p}{p-2}} \cdot
		\left(\sup_{\sigma > 0} \frac{1 + \sigma \eta }
		{(1 + 2 b_- \sigma^{\frac{p}{2}} + \sigma^p)^{\frac{1}{p}} (1 + \eta^2 q)^{\frac{1}{2}}}
		\right)^\frac{2p}{p-2}
		\\
		&\overset{\eqref{eq_prf_thm2a}}{=}
		c_\mu
		\cdot
		\left(\sup_{\sigma > 0} \frac{1 + \sigma \eta }
		{(1 + 2 b_- \sigma^{\frac{p}{2}} + \sigma^p)^{\frac{1}{p}} (1 + \eta^2 q)^{\frac{1}{2}}}\right)^\frac{2p}{p-2}
		< c_\mu.
	\end{align*}
	The latter estimate holds for sufficiently small positive $\eta$ because we have, 
	with $\tilde{b}_- := \min\{ 1, b_- \} > 0$ and Taylor's Theorem,
	\begin{align*}
		\sup_{\sigma > 0} \: 
		\frac{1 + \sigma \eta}
		{(1 + 2 b_- \sigma^\frac{p}{2} + \sigma^p)^\frac{1}{p}(1 + \eta^2 q)^{\frac{1}{2}}}
		&\leq
		\sup_{\sigma > 0} \: 
		\frac{1 + \sigma \eta}
		{(1 + \tilde{b}_- \sigma^\frac{p}{2})^\frac{2}{p}(1 + \eta^2 q)^{\frac{1}{2}}}
		\\
		&=
		\frac{\big( 1 + \eta^\frac{p}{p-2} \tilde{b}_-^{-\frac{2}{p-2}} \big)^\frac{p-2}{p}}
		{(1 + \eta^2 q)^{\frac{1}{2}}}
		\\
		&=
		\frac{1 + \frac{p-2}{p} \eta^\frac{p}{p-2} 
		\tilde{b}_-^{-\frac{2}{p-2}} + o\big(\eta^\frac{p}{p-2}\big)}
		{1 + \frac{1}{2} \eta^2 q + o(\eta^2)}
		\\
		&= 
		1 - \frac{1}{2} \eta^2 q + o(\eta^2)
		& \text{as } \eta \searrow 0
	\end{align*}
	where we used that $\frac{p}{p-2} > 2$ since $2 < p < 4$. 
	We have shown that $c_{\mu\nu} < c_\mu$. Similarly, one proves that $c_{\mu\nu} < c_\nu$. 			
	Lemma~\ref{lem_scalarmp} (ii) implies that $J_{\mu\nu}$ cannot have a semitrivial dual ground state. 
\end{proof}

The proof of Theorem~\ref{thm_c_mumu_plus-1} (b) is based on a continuity argument which requires additional knowledge of the scalar case for $a \equiv 1$. Here we let
\begin{align*}
	D_\lambda (\bar{w}) := \frac{p-2}{2p}
	\left(\frac{\left( \int_{\R^N} |\bar{w}|^{p^\prime} 
	\: \mathrm{d}x \right)^\frac{1}{p^\prime}}
	{\left(\int_{\R^N} \bar{w} \Psi_\lambda \ast \bar{w} \: \mathrm{d}x \right)_+^\frac{1}{2}}
	\right)^\frac{2p}{p-2},
	\qquad
	d_\lambda
	:=
	\inf_{\bar{w} \in \L{p^\prime}} D_\lambda (\bar{w})
\end{align*}
and, in view of definition~\eqref{eq_Emu}, immediately note that
\begin{equation}\label{eq_cmumu-a}
	a_+^{-\frac{2}{p-2}} d_\lambda 
	\leq c_\lambda
	\leq a_-^{-\frac{2}{p-2}} d_\lambda.
\end{equation}
Remark~\ref{rmk_scalar} guarantees that the functional $I_1$ with $a(x) \equiv 1$ admits a dual ground state $\bar{z} \in \L{p^\prime}$ which, by the remarks following Lemma~\ref{lem_mplevel}, is a minimizer of the functional $D_1$. We fix such a minimizer $\bar{z}$ and introduce for $\lambda > 0$ the rescaled functions
\begin{align}\label{eq_zRescale}
	\bar{z}_\lambda \in \L{p^\prime},
	\qquad
	\bar{z}_\lambda(x) := \lambda^\frac{N+2}{4} \bar{z}(\sqrt{\lambda} \, x),
	\qquad x \in \R^N.
\end{align}
Then $\bar{z}_\lambda$ is a minimizer of the functional $D_\lambda$, 
and we have
\begin{equation}\label{eq_LemContC}
	d_\lambda = \lambda^{\frac{p}{p-2} - \frac{N}{2}} \cdot d_1
	= \lambda^{\frac{p}{p-2} - \frac{N}{2}}	\cdot \frac{p-2}{2p} \left(
		\frac{\left[\int_{\R^N} |\bar{z}|^{p'} \: \mathrm{d}x\right]^\frac{1}{p'}}
		{\left[\int_{\R^N} \bar{z} \Psi_1 \ast \bar{z} \: \mathrm{d}x\right]^\frac{1}{2}}
	\right)^\frac{2p}{p-2}.
\end{equation}
The proof of~\eqref{eq_LemContC} is based on the observation
\begin{equation}\label{eq_LemContB}
	\Psi_\lambda(x)  = \lambda^\frac{N-2}{2} \Psi_1 \left( \sqrt{\lambda} x \right)
	\quad
	\text{and}
	\quad
	\int_{\R^N} \bar{z}_\lambda \Psi_\lambda \ast \bar{z}_\lambda \: \mathrm{d}x 
		= \int_{\R^N} \bar{z} \Psi_1 \ast \bar{z} \: \mathrm{d}x
\end{equation}
for $x \in \R^N$ and $\lambda > 0$, see~\eqref{eq_phi}.

\begin{proof}[\textbf{Proof of Theorem~\ref{thm_c_mumu_plus-1} (b)}]
	We aim to prove $c_{\mu\nu} < \min\{c_\mu, c_\nu\}$ 
for sufficiently small values of $\left| \sqrt{\frac{\mu}{\nu}} - 1 \right|$, 
which again yields the assertion when applying	 Lemma~\ref{lem_scalarmp} (ii).
With $\bar{z}, \bar{z}_{\mu}, \bar{z}_{\nu} \in \L{p'}$ as above, we estimate as follows:
	\begin{align*}
	c_{\mu\nu}
	&\leq
	F_{\mu\nu}\left(\bar{z}_\mu, \bar{z}_\nu\right)
	\\*
	&\: =
	\frac{p-2}{2p}
	\left(\frac{\left[ \int_{\R^N} p^\prime \, 
	h\left(x, \bar{z}_\mu, \bar{z}_\nu\right) 	
	\: \mathrm{d}x \right]^\frac{1}{p^\prime}}
	{\left[ \int_{\R^N} \bar{z}_\mu 	\Psi_\mu \ast \bar{z}_\mu \: \mathrm{d}x
	+
	\int_{\R^N} \bar{z}_\nu 	\Psi_\nu \ast \bar{z}_\nu \: \mathrm{d}x \right]_+^\frac{1}{2}}
	\right)^\frac{2p}{p-2}
	\\	
	&\: \overset{\eqref{eq_LemContB}}{=}
	\frac{p-2}{2p}
	\left(\frac{\left[ \int_{\R^N} p^\prime \, 
	h\left(x, \bar{z}_\mu, \bar{z}_\nu\right) 	
	\: \mathrm{d}x \right]^\frac{1}{p^\prime}}
	{\left[ 2 \, \int_{\R^N} \bar{z} \Psi_1 \ast \bar{z} \: \mathrm{d}x \right]^\frac{1}{2}}
	\right)^\frac{2p}{p-2}
	\\
	&\: \overset{\eqref{eq_LemContC}}{=}
	\frac{d_1}{2^\frac{p}{p-2}} 
	\left(\frac{ \int_{\R^N} p^\prime \, 
	h\left(x, \bar{z}_\mu, \bar{z}_\nu\right) 	
	\: \mathrm{d}x }
	{\int_{\R^N} |\bar{z}|^{p'} \: \mathrm{d}x}
	\right)^\frac{2(p-1)}{p-2}
	\\
	&\: \overset{\eqref{eq_mono-a}}{\leq}
	\frac{d_1}{2^\frac{p}{p-2}} 
	\left(\frac{ a_-^{1 - p^\prime} \int_{\R^N} p^\prime \, 
	h_-\left(\bar{z}_\mu, \bar{z}_\nu\right) 	
	\: \mathrm{d}x }
	{\int_{\R^N} |\bar{z}|^{p'} \: \mathrm{d}x}
	\right)^\frac{2(p-1)}{p-2}
	\\
	&\: \overset{\text{Lem.}\ref{lem_prop_hb}(d)}{\leq}
	\frac{d_1}{2^\frac{p}{p-2} a_-^\frac{2}{p-2}} 
	\left(\frac{\int_{\R^N} 
	h_-\left(\bar{z}_\mu, \bar{z}_\nu\right) 	
	\: \mathrm{d}x}
	{\frac{1}{2} (1 + b_-)^\frac{1}{p-1}  \int_{\R^N} h_-(\bar{z}, \bar{z}) \: \mathrm{d}x}
	\right)^\frac{2(p-1)}{p-2}
	\\
	&\: =
	d_1 \cdot  \left( \frac{2^\frac{p-2}{2}}{1 + b_-} \, \frac{1}{a_-}  \right)^\frac{2}{p-2}
	\left(\frac{\int_{\R^N} 
	h_-\left(\bar{z}_\mu, \bar{z}_\nu\right) 		
	\: \mathrm{d}x}
	{\int_{\R^N} h_-(\bar{z}, \bar{z}) \: \mathrm{d}x}
	\right)^\frac{2(p-1)}{p-2}.
\end{align*}
We now introduce $\lambda := \sqrt{\frac{\mu}{\nu}}$.
Then, with $h_-(\alpha \bar{s}, \alpha \bar{t}) = |\alpha|^{p'} h_-(\bar{s}, \bar{t})$ (see equation~\eqref{eq_hb}) and substitution:
\begin{align*}
	\left( \int_{\R^N} h_-\left( \bar{z}_\mu, \bar{z}_\nu\right) 		
	\: \mathrm{d}x \right)^\frac{2(p-1)}{p-2}
	&= 
	\left( \int_{\R^N} h_-\left( \mu^\frac{N+2}{4} \bar{z}(\sqrt{\mu} x), 
	\nu^\frac{N+2}{4} \bar{z}(\sqrt{\nu} x)\right) 		
	\: \mathrm{d}x \right)^\frac{2(p-1)}{p-2}
	\\
	&= 
	\left( \int_{\R^N} h_-\left( \lambda^\frac{N+2}{2} \nu^\frac{N+2}{4} \bar{z}(\lambda \sqrt{\nu} x), 
	\nu^\frac{N+2}{4} \bar{z}(\sqrt{\nu} x)\right) 		
	\: \mathrm{d}x \right)^\frac{2(p-1)}{p-2}
	\\
	& =
	\nu^{\frac{p}{p-2} - \frac{N}{2}}
	\left( \int_{\R^N} h_-\left( \lambda^\frac{N+2}{2} \bar{z}(\lambda y), 
	\bar{z}(y)\right) 	\: \mathrm{d}y \right)^\frac{2(p-1)}{p-2}
	\\
	& =
	\nu^{\frac{p}{p-2} - \frac{N}{2}}
	\left( \int_{\R^N} h_-\left( \bar{z}_{\lambda^2}, \bar{z}\right) 	\: \mathrm{d}y \right)^\frac{2(p-1)}{p-2}
\end{align*}
We insert this into the previous estimate and find
\begin{align*}
	c_{\mu\nu}
	&\leq
	d_1\: \nu^{\frac{p}{p-2} - \frac{N}{2}} 
	\cdot  \left( \frac{2^\frac{p-2}{2}}{1 + b_-} \, \frac{1}{a_-}  \right)^\frac{2}{p-2}
	\left(\frac{\int_{\R^N} 
	h_-\left(\bar{z}_{\lambda^2}, \bar{z}\right) 		
	\: \mathrm{d}x}
	{\int_{\R^N} h_-(\bar{z}, \bar{z}) \: \mathrm{d}x}
	\right)^\frac{2(p-1)}{p-2}
	\\
	&\overset{\eqref{eq_LemContC}}{=}
	d_\nu \cdot  \left( \frac{2^\frac{p-2}{2}}{1 + b_-} \, \frac{1}{a_-}  \right)^\frac{2}{p-2}
	\left(\frac{\int_{\R^N} 
	h_-\left(\bar{z}_{\lambda^2}, \bar{z}\right) 		
	\: \mathrm{d}x}
	{\int_{\R^N} h_-(\bar{z}, \bar{z}) \: \mathrm{d}x}
	\right)^\frac{2(p-1)}{p-2}
	\\
	&\overset{\eqref{eq_cmumu-a}}{\leq}
	c_\nu \cdot  \left( \frac{2^\frac{p-2}{2}}{1 + b_-} \, \frac{a_+}{a_-}  \right)^\frac{2}{p-2}
	\left(\frac{\int_{\R^N} 
	h_-\left(\bar{z}_{\lambda^2}, \bar{z}\right) 		
	\: \mathrm{d}x}
	{\int_{\R^N} h_-(\bar{z}, \bar{z}) \: \mathrm{d}x}
	\right)^\frac{2(p-1)}{p-2}.
\end{align*}
Similarly,
\begin{align*}
	c_{\mu\nu}
	&\leq
	c_\mu \cdot  \left( \frac{2^\frac{p-2}{2}}{1 + b_-} \, \frac{a_+}{a_-}  \right)^\frac{2}{p-2}
	\left(\frac{\int_{\R^N} 
	h_-\left(\bar{z}_{\lambda^{-2}}, \bar{z}\right) 		
	\: \mathrm{d}x}
	{\int_{\R^N} h_-(\bar{z}, \bar{z}) \: \mathrm{d}x}
	\right)^\frac{2(p-1)}{p-2}.
\end{align*}
Notice that the terms on the right 
depend continuously on the parameter $\lambda$ since $\lambda \mapsto 
\lambda^\frac{N+2}{2} \bar{z}(\lambda \:\cdot\:)$ is continuous in $\L{p^\prime}$. 
Hence,
\begin{align*}
	&\frac{\int_{\R^N} 
	h_-\left(\bar{z}_{\lambda^{\pm 2}}, \bar{z}\right) 		
	\: \mathrm{d}x}
	{\int_{\R^N} h_-(\bar{z}, \bar{z}) \: \mathrm{d}x}
	\to 1
	\qquad \text{as } \lambda \to 1.
\end{align*}
As we have assumed $\frac{2^\frac{p-2}{2}}{1 + b_-} \cdot \frac{a_+}{a_-} < 1$,
we find $\delta > 0$ such that $|\lambda - 1 | < \delta$ implies
\begin{align*}
	\left( \frac{2^\frac{p-2}{2}}{1 + b_-} \, \frac{a_+}{a_-}  \right)^\frac{2}{p-2}
	\left(\frac{\int_{\R^N} 
	h_-\left(\bar{z}_{\lambda^{\pm 2}}, \bar{z}\right) 		
	\: \mathrm{d}x}
	{\int_{\R^N} h_-(\bar{z}, \bar{z}) \: \mathrm{d}x}
	\right)^\frac{2(p-1)}{p-2}
	< 1
	\quad
	\text{and hence}
	\quad
	c_{\mu\nu} < \min\{c_\mu, c_\nu\}.
\end{align*}
Lemma~\ref{lem_scalarmp} (ii) ensures that, for such $\mu$ and $\nu$, every dual ground state is fully nontrivial.
\end{proof}
\begin{proof}[\textbf{Proof of Theorem~\ref{thm_c_mumu_plus-2}}]
	We consider a dual ground state $(\bar{u}, \bar{v}) \in \LL{p^\prime}$ of the functional $J_{\mu\nu}$, hence
	a minimizer of $F_{\mu\nu}$, and w.l.o.g. $\bar{u} \neq 0$.
	We write $\norm{a^{-1/p} \bar{v}}_{p^\prime} = \eta_0 \norm{a^{-1/p} \bar{u}}_{p^\prime}$
	for some $\eta_0 \geq 0$ and aim to show that necessarily $\eta_0 = 0$, hence $\bar{v} = 0$. \\ 
	Recalling that $c_{\mu}, c_{\nu}$ 
	are the minima of $E_{\mu}, E_{\nu}$, respectively, we estimate with~\eqref{eq_Emu}
	\begin{align*}
		\left(\int_{\R^N} \bar{u} \Psi_\mu \ast \bar{u} \: \mathrm{d}x\right)_+
		&\leq \left( \frac{2p}{p-2} c_\mu\right)^{-\frac{p-2}{p}} \, \norm{a^{-1/p} \bar{u}}_{p^\prime}^2,
		\\
		\left(\int_{\R^N} \bar{v} \Psi_\nu \ast \bar{v} \: \mathrm{d}x\right)_+
		&\leq \left( \frac{2p}{p-2} c_\nu\right)^{-\frac{p-2}{p}} \, \norm{a^{-1/p} \bar{v}}_{p^\prime}^2.
	\end{align*}
	In the appendix, we will prove the estimate
	\begin{equation}\label{eq_auxiliary_h-below}
		\int_{\R^N} h(x, \bar{u}, \bar{v}) \: \mathrm{d}x 
		\geq h_+\left(\norm{a^{-1/p} \, \bar{u}}_{p^\prime}, 
		\norm{a^{-1/p} \, \bar{v}}_{p^\prime}\right).
	\end{equation}
	Both inequalities combined, we have
	\begin{align*}
		c_{\mu\nu} 
		&= F_{\mu\nu}(\bar{u}, \bar{v})
		\\		
		&\geq
		 \frac{p-2}{2p} \frac{\bigg[ 
		 p^\prime \, h_+\left(\norm{a^{-1/p} \, \bar{u}}_{p^\prime}, 
		\eta_0 \norm{a^{-1/p} \, \bar{u}}_{p^\prime}\right) 
		\bigg]^\frac{2(p-1)}{p-2}}
		{\bigg[  \left( \frac{2p}{p-2} c_\mu\right)^{-\frac{p-2}{p}} \, \norm{a^{-1/p} \bar{u}}_{p^\prime}^2
		+  \left( \frac{2p}{p-2} c_\nu\right)^{-\frac{p-2}{p}} \, \eta_0^2 \norm{a^{-1/p} \bar{u}}_{p^\prime}^2
		\bigg]^\frac{p}{p-2}}
		\\		
		&\overset{\text{Lem.} \ref{lem_prop_hb} (a)}{=}
		 \frac{\bigg[ 
		 p^\prime \, \norm{a^{-1/p} \, \bar{u}}_{p^\prime}^{p^\prime} \, 
		 h_+\left(1, \eta_0\right) 
		\bigg]^\frac{2(p-1)}{p-2}}
		{\bigg[  \left( \left(  c_\mu\right)^{-\frac{p-2}{p}}	+  \left(  c_\nu\right)^{-\frac{p-2}{p}} \, \eta_0^2 \right) 		
		\, \norm{a^{-1/p} \bar{u}}_{p^\prime}^2
		\bigg]^\frac{p}{p-2}}
		\\
		&\geq
		\min\{ c_\mu, c_\nu \} \, \left( \frac{\big[ 
		p^\prime \,  h_+\left( 1, \eta_0 \right) 
		\big]^\frac{1}{p^\prime}}
		{\big[ 1 + \eta_0^2  \big]^\frac{1}{2}} \right)^\frac{2p}{p-2}
		\\
		&\overset{\eqref{eq_hb}}{=}
		\min\{ c_\mu, c_\nu \} \: \left( \sup_{\sigma > 0} \frac{(1 + \sigma \eta_0)}
		{(1 + 2 b_+ \sigma^\frac{p}{2} + \sigma^p)^\frac{1}{p}(1 + \eta_0^2)^\frac{1}{2}}
		\right)^\frac{2p}{p-2}
		\\
		& \geq
		\min\{ c_\mu, c_\nu \} \: \left(\frac{(1 + \eta_0^2)^\frac{1}{2}}
		{(1 + 2 b_+ \eta_0^\frac{p}{2} + \eta_0^p)^\frac{1}{p}}
		\right)^\frac{2p}{p-2}
		\\
		& \geq
		\min\{ c_\mu, c_\nu \} \: \left(\frac{(1 + \eta_0^2)^\frac{1}{2}}
		{(1 + (2^\frac{p}{2} - 2) \eta_0^\frac{p}{2} + \eta_0^p)^\frac{1}{p}}
		\right)^\frac{2p}{p-2}
		\\
		& \geq \min\{ c_\mu, c_\nu \}.
	\end{align*}
	The latter estimate holds since
	\begin{equation}\label{eq_auxiliary_mini}
		\forall \: \eta \geq 0
		\qquad
		\frac{(1 + \eta^2)^\frac{1}{2}}
		{(1 + (2^\frac{p}{2} - 2) \eta^\frac{p}{2} + \eta^p)^\frac{1}{p}} 
		\geq 1,
	\end{equation}
	which we will prove in the appendix.
	Lemma~\ref{lem_scalarmp} (ii) yields $c_{\mu\nu} \leq \min\{ c_\mu, c_\nu \}$,
	and thus we have $c_{\mu\nu} = \min\{ c_\mu, c_\nu \}$ and equality must 
	hold in all above estimates.
	But then, since we assume $b_+ < 2^\frac{p-2}{2} - 1$, we infer $\eta_0 = 0$. 	
	Hence, the dual ground state is semitrivial with $\bar{v} = 0$.
\end{proof}
\begin{proof}[\textbf{Proof of Corollary~\ref{cor_plus}}]
	The previously proved Theorems cover most cases:
If $2 < p < 4 \text{ and } b > 0$, Theorem~\ref{thm_c_mumu_plus-1} (a) states that every dual ground state of $J_{\mu\mu}$ is fully nontrivial; so does Theorem~\ref{thm_c_mumu_plus-1} (b) in case $p \geq 4 \text{ and } b > 2^{\frac{p-2}{2}} - 1$. (Notice that we assume $\nu = \mu$.) If, however, $p \geq 4 \text{ and } 0 \leq b < 2^{\frac{p-2}{2}} - 1$, Theorem~\ref{thm_c_mumu_plus-2} assures that dual ground states of $J_{\mu\mu}$ are semitrivial.
So only two cases remain open.

\absatz

\textit{Assume $2 < p < 4$ and $b = 0$.}

The proof then follows the lines of that of Theorem~\ref{thm_c_mumu_plus-2}.
Considering a dual ground state $(\bar{u}, \bar{v}) \in \LL{p'}$ of $J_{\mu\nu}$ and assuming  $\norm{a^{-1/p} \bar{v}}_{p^\prime} = \eta_0 \norm{a^{-1/p} \bar{u}}_{p^\prime}$, 
we again aim to prove that $\eta_0 = 0$.
The same estimate as in the previous proof yields
\begin{align*}
	c_{\mu\nu} 
		&\geq \min\{ c_\mu, c_\nu \} \: \left(\sup_{\sigma > 0} \frac{(1 + \sigma \eta_0)}
		{(1 + \sigma^p)^\frac{1}{p}(1 + \eta_0^2)^\frac{1}{2}} \right)^\frac{2p}{p-2}
		\\
		&\geq
		\min\{ c_\mu, c_\nu \} \: \left(\frac{(1 + \eta_0^{p'})^\frac{1}{p'}}
		{(1 + \eta_0^2)^\frac{1}{2}}\right)^\frac{2p}{p-2}
		\\
		&\overset{(p' < 2)}{\geq}
		\min\{ c_\mu, c_\nu \}
\end{align*}
with equality if and only if $\eta_0 = 0$. Since $c_{\mu\nu} \leq \min\{ c_\mu, c_\nu \}$ by Lemma~\ref{lem_scalarmp} (i), this implies $\eta_0 = 0$.

\absatz

\textit{Assume $p \geq 4$ and $b = 2^\frac{p-2}{2} - 1$.}

In this case, one can show as in the proof of Theorem~\ref{thm_c_mumu_plus-2} that we have $c_{\mu\mu} = c_\mu$. 
For any scalar dual ground state $\bar{w} \in \L{p^\prime}$ of $I_\mu$, we calculate
\begin{align*}
	F_{\mu\mu}(\bar{w}, \bar{w})
	&\overset{\eqref{eq_Fmunu}}{=}
	\frac{p-2}{2p} \left(
	\frac{\left[\int_{\R^N} p^\prime \, h(x, \bar{w}, \bar{w}) \: \mathrm{d}x\right]^\frac{1}{p^\prime}}
	{\left[\int_{\R^N} \bar{w} \Psi_\mu \ast \bar{w} 
	+ \bar{w} \Psi_\mu \ast \bar{w} \: \mathrm{d}x\right]_+^\frac{1}{2}}
	\right)^\frac{2p}{p-2}
	\\
	&\overset{\text{Lem. }\ref{lem_prop_hb} (d)}{=}
	\frac{p-2}{2p} \left(
	\frac{\left[\int_{\R^N} 2 (1 + b)^{1-p'} a(x)^{1-p'} |\bar{w}(x)|^{p'} \: \mathrm{d}x\right]^\frac{1}{p^\prime}}
	{\left[2 \, \int_{\R^N} \bar{w} \Psi_\mu \ast \bar{w}\right]_+^\frac{1}{2}}
	\right)^\frac{2p}{p-2}
	\\
	&=
	\frac{2}{(1+b)^\frac{2}{p-2}} \:
	\frac{p-2}{2p} \left(
	\frac{\left[\int_{\R^N} a(x)^{1-p'} |\bar{w}(x)|^{p'} \: \mathrm{d}x\right]^\frac{1}{p^\prime}}
	{\left[\int_{\R^N} \bar{w} \Psi_\mu \ast \bar{w}\right]_+^\frac{1}{2}}
	\right)^\frac{2p}{p-2}
	\\
	&\overset{\eqref{eq_Emu}}{=}
	\frac{2}{(1+b)^\frac{2}{p-2}} \: E_\mu (\bar{w})
	\\
	&= E_\mu (\bar{w})
	\\
	&= c_\mu
\end{align*}
and $F_{\mu\mu}(\bar{w}, 0) = E_\mu(\bar{w}) = c_\mu$. Hence,
$F_{\mu\mu}(\bar{w}, \bar{w}) = F_{\mu\mu}(\bar{w}, 0) = c_{\mu\mu}$
and by Lemma~\ref{lem_mplevel}, this provides (up to multiplication with suitable constants) both a semitrivial and a fully nontrivial dual ground state of $J_{\mu\mu}$. 
\end{proof}

\section{Appendix}

\subsection{Proof of Proposition~\ref{prop_fbhb}}
	
	Fix $x \in \R^N$ and recall for $s, t \in \R$ 
	\begin{align*}
		f(x, s, t) = \frac{a(x)}{p} \left( |s|^p + 2 b(x) \: |s|^\frac{p}{2} |t|^\frac{p}{2} + |t|^p \right).
	\end{align*}
	Differentiability and co-finiteness of $f(x, \:\cdot\:, \:\cdot\: )$ 
	are a straightforward consequence of the assumption $p > 2$. 
	We will show below that $f(x, \:\cdot\:, \:\cdot\: )$ is strictly convex;
	with that, the existence and the asserted properties of the Legendre transform $h(x, \:\cdot\:, \:\cdot\: )$ 
	of $f(x, \:\cdot\:, \:\cdot\: )$	are guaranteed by Theorem~\ref{thm_diffb}.
	To verify strict convexity, we show that, for all $s_1, s_2, t_1, t_2 \in \R$
	with $s_2 \neq 0$ or $t_2 \neq 0$,
\begin{align}\label{eq_convex}
	f(x, s_1 + s_2, t_1 + t_2) > f(x, s_1, t_1) 
	+ s_2 \: \partial_s f(x, s_1, t_1) 
	+ t_2 \: \partial_t f(x, s_1, t_1).
\end{align}	
	We denote the difference by 
	\begin{align*}
		\mathcal{I} := 
		f(x, s_1 + s_2, t_1 + t_2) - \bigg[ 
		f(x, s_1, t_1) 
	+ s_2 \: \partial_s f(x, s_1, t_1) 
	+ t_2 \: \partial_t f(x, s_1, t_1)
	\bigg].
	\end{align*}
	So if we prove $\mathcal{I} > 0$, we conclude~\eqref{eq_convex}. 
	We introduce the line segment
	\begin{align*}
		\ell := \left\{ (s_1, t_1) + \theta (s_2, t_2) \in \R^2:
		\: \: 0 \leq \theta \leq 1 \right\}.
	\end{align*}		
	
	\begin{itemize}
	\item[(I)] 
	Let us assume that $\ell$ is a subset of either of the sets	
	\begin{equation}\label{eq_append_sets}
	\begin{split}
		&\{ (s, 0)\in \R^2: \: s \in \R \},
		\quad
		\{ (0, t)\in \R^2: \: t \in \R \},
		\\
		&\{ (s, s)\in \R^2: \: s \in \R \}, 
		\quad
		\{ (s, -s)\in \R^2: \: s \in \R \}.
	\end{split}
	\end{equation}
	We then conclude $\mathcal{I} > 0$ since the functions
	\begin{align*}
		&s \mapsto f(x, s, 0) = \frac{a(x)}{p} |s|^p,
		\quad
		&&t \mapsto f(x, 0, t) = \frac{a(x)}{p} |t|^p,
		\\
		&s \mapsto f(x, s, s) = \frac{2a(x)(1 + b(x))}{p} |s|^p, 
		\quad
		&&s \mapsto f(x, s, -s) = \frac{2a(x)(1 + b(x))}{p} |s|^p,
	\end{align*}
	respectively, are strictly convex.
	
	\item[(II)]
	We now assume that $\ell$ intersects none of the sets in~\eqref{eq_append_sets}. 
	Then $f$ is twice continuously differentiable in a neighborhood of $\ell$, and the Fundamental Theorem
	of Calculus yields the integral representation 
	\begin{equation}\label{eq_append_convex}
		\mathcal{I} = 
		\int_0^1 \int_0^1 \tau \:
		( s_2, t_2)
		D^2_{s,t} f(x, s_1 + \tau\sigma s_2, t_1 + \tau\sigma t_2)
		\begin{pmatrix} s_2 \\ t_2	\end{pmatrix} 
		\: \mathrm{d}\sigma \: \mathrm{d}\tau.
	\end{equation}
	We show that the Hessian $D^2_{s, t} f(x, s, t)$ is strictly positive definite for all $(s, t) \in \ell$.

	\absatz	
	
	\noindent
	Let $(s, t) \in \ell$, i.e. in particular $s \neq 0$, $t \neq 0$ and $|s| \neq |t|$.
	Recall that we assume $a(x) > 0$ and $0 \leq b(x) \leq p-1$.
	We calculate the trace and the determinant of the Hessian:
	\begin{align*}
		&\text{tr } D^2_{s, t} f(x, s, t)
		\\
		&\quad
		= a(x) (p-1) \left(|s|^{p-2} + |t|^{p-2}\right) 
		+ a(x) \frac{b(x)}{2} (p-2) \left(|s|^{\frac{p}{2} - 2} |t|^\frac{p}{2} 
		+ |t|^{\frac{p}{2} - 2} |s|^\frac{p}{2} \right),
		\\
		&\det D^2_{s, t} f(x, s, t)
		\\
		&\quad 
		=	a(x)^2 (p-1) \left[ 
		( p-1-b(x)^2 ) |s|^{\frac{p}{2}} |t|^{\frac{p}{2}}
		+ \frac{b(x)}{2} (p-2) \big( |s|^p + |t|^p \big)  
		\right] |s|^{\frac{p}{2} - 2} |t|^{\frac{p}{2} - 2} .
	\end{align*}
	Since $a(x) > 0$ $b(x) \geq 0$, $s \neq 0$ and $t \neq 0$, we always have $\text{tr } D^2_{s, t} f(x, s, t) > 0$. 
	If $0 \leq b(x) \leq \sqrt{p - 1}$, we infer $\det D^2_{s, t} f(x, s, t) > 0$ and 
	hence $D^2_{s, t} f(s, t)$ is strictly positive semidefinite. 
	Else if $\sqrt{p-1} < b(x) \leq p-1$, we recall that $|s| \neq |t|$ by assumption on $\ell$,
	which gives the strict  estimate $|s|^\frac{p}{2} |t|^\frac{p}{2} < \frac{1}{2} \big( |s|^p + |t|^p \big)$.
	Thus,
	\begin{align*}
		\det D^2_{s, t} f(x, s, t)
		& >
		a(x)^2 \frac{p-1}{2} \big((p-1-b(x)^2) + b(x)(p - 2)\big) 
		\big( |s|^p + |t|^p \big) |s|^{\frac{p}{2} - 2} |t|^{\frac{p}{2} - 2}
		\\
		&=
		a(x)^2 \frac{(p-1)(b(x)+1)}{2} (p-1-b(x)) 
		\big( |s|^p + |t|^p \big) |s|^{\frac{p}{2} - 2} |t|^{\frac{p}{2} - 2}
		\\
		& \geq 0,
	\end{align*}
	which proves strict positive definiteness of $D^2_{s, t} f(x, s, t)$.
	
	\item[(III)]
	Finally, in all remaining cases, $\ell$ intersects the sets of~\eqref{eq_append_sets} in at most 
	finitely many points. Then still, the integral in~\eqref{eq_append_convex} converges,
	the integral representation from~\eqref{eq_append_convex} holds 
	and the previous step gives $\mathcal{I} > 0$.
	\end{itemize}
	
	Hence, $f(x, \:\cdot\:, \:\cdot\: )$ is strictly convex, which concludes the proof.
\qed

\subsection{Proof of Lemma~\ref{lem_prop_hb}}
For $\bar{s}, \bar{t} \in \R$, we recall the definition of the Legendre transform:
\begin{align*}
	h(x, \bar{s}, \bar{t}) = \sup_{s, t \in \R} \bigg( 
		s\bar{s} + t\bar{t} - f(x, s, t)
	\bigg)
\end{align*}
where $f(x, s, t) = \frac{a(x)}{p} \left(|s|^p + 2b(x) \: |s|^\frac{p}{2}|t|^\frac{p}{2} + |t|^p\right)$.
We note that, since $f(x, s, t) \geq 0$, this immediately yields $h(x, 0, 0) = 0$.
\begin{enumerate}[label=(\alph*)]
\item
We assume w.l.o.g. that $\bar{s} \neq 0$. 
With that, 
	\begin{align*}
		h(x, \bar{s}, \bar{t}) 
		&= \sup_{s, t \in \R} \bigg[ s \bar{s} + t \bar{t}	- f(x, s, t) \bigg]
		\overset{f(x, s, t) \geq 0}{=} \sup_{s, t > 0} \bigg[ s |\bar{s}| +  t |\bar{t}|	- f(x, s, t) \bigg]
		\\
		&= \sup_{s, \sigma > 0} \bigg[ s(|\bar{s}| + \sigma |\bar{t}|)	- s^p f(x, 1, \sigma) \bigg]
		= \sup_{\sigma > 0} \: \: \frac{1}{p^\prime} 
		\left( \frac{(|\bar{s}| + \sigma |\bar{t}|)^p}{p f(x, 1, \sigma)} \right)^\frac{1}{p-1}
		\\
		&	= \frac{a(x)^{1 - p^\prime}}{p^\prime} \left[\sup_{\sigma > 0} 
		\frac{|\bar{s}| + \sigma |\bar{t}|}{ (1 + 2b(x) \, \sigma^\frac{p}{2} + \sigma^p)^\frac{1}{p}}\right]^{p^\prime}
	\end{align*}
where the supremum with respect to $s > 0$ has been evaluated explicitly.
\item
 This is a direct consequence of the symmetry of $f(x, \:\cdot\:, \:\cdot\: )$, i.e. $f(x, s, t) = f(x, t, s)$
 and of the fact that $f(x, -s, t) = f(x, s, t)$, respectively, for all $s, t \in \R$.
\item
As a consequence of part (a), we have 
$
	h(x, \alpha \bar{s}, \alpha \bar{t}) = \alpha^{p'} h(x, \bar{s}, \bar{t})
$
for $\alpha > 0$.
We differentiate with respect to $\alpha$ and find
\begin{align*}
	\nabla_{\bar{s}, \bar{t}} h(x, \alpha \bar{s}, \alpha \bar{t}) \cdot \svec{\bar{s}}{\bar{t}} 
	= p' \alpha^{p' - 1} h(x, \bar{s}, \bar{t}).
\end{align*}
Evaluating the latter identity at $\alpha = 1$, the assertion of (c) is proved.
\item
We only prove the second identity, the first one can be shown in the same way. By direct computation we find
$\nabla_{s, t} f(x, s, s) = a(x) (1 + b(x)) |s|^{p-2} s\:  (1, 1)$ for $s \in \R$. Recalling that $\nabla_{\bar{s}, \bar{t}} h(x, \:\cdot\:, \:\cdot\: )$ is a diffeomorphism on $\R^2$ with inverse $\nabla_{s, t} f(x, \:\cdot\:, \:\cdot\: )$, this  implies $\nabla_{\bar{s}, \bar{t}} h(x, \bar{s}, \bar{s}) = \big(a(x) (1 + b(x))\big)^{-{(p^\prime - 1)}} |\bar{s}|^{p^\prime - 2} \bar{s} \:  (1, 1)$, and hence using (c)
\begin{align*}
	h(x, \bar{s}, \bar{s}) = 
	\frac{1}{p^\prime} \nabla_{\bar{s}, \bar{t}} h(x, \bar{s}, \bar{s}) \cdot \svec{\bar{s}}{\bar{s}} =
	\frac{2}{p^\prime} \big(a(x) (1 + b(x))\big)^{1-p^\prime} |\bar{s}|^{p^\prime}.
\end{align*}
\item
We have by definition of the Legendre transform and due to $a(x), b(x) \geq 0$
\begin{align*}
	h(x, \bar{s}, \bar{t}) 
	&
	= \sup_{s, t \in \R} \bigg( 
		s\bar{s} + t\bar{t} - \frac{a(x)}{p} \left(
			|s|^p + 2b(x) \: |s|^\frac{p}{2}|t|^\frac{p}{2} + |t|^p
		\right)
	\bigg)
	\\
	&
	\leq \sup_{s, t \in \R} \bigg( 
		s\bar{s} + t\bar{t} - \frac{a(x)}{p} \left(
			|s|^p  + |t|^p
		\right)
	\bigg)
	= \frac{1}{p^\prime} a(x)^{1 - p^\prime} \big( |\bar{s}|^{p^\prime} + |\bar{t}|^{p^\prime} \big)
\end{align*}
where we calculated the latter supremum explicitly.
On the other hand, defining $s_x \in \R$ via
\begin{align*}
		s_x &:= (a(x)(1+b(x)))^{1 - p^\prime} \cdot |\bar{s}|^{p^\prime - 2} \bar{s}
\end{align*}	
we notice that $s_x$ maximizes the map 
$\R \to \R, \: s \mapsto s \bar{s} - \frac{1}{p} a(x) (1 + b(x)) |s|^p$
and that $\bar{s} = a(x)(1 + b(x)) |s_x|^{p-2} s_x$. Defining $t_x \in \R$ similarly, we estimate 
\begin{align*}
		&\frac{1}{p^\prime} (a(x)(1+b(x)))^{1 - p^\prime} \left( |\bar{s}|^{p^\prime} + |\bar{t}|^{p^\prime}  \right)
		\\
		& \quad =
		\left( 1 - \frac{1}{p} \right) (s_x \bar{s} + t_x \bar{t}) 
		\\
		& \quad \leq
		s_x \bar{s} + t_x \bar{t} - \frac{1}{p} \left( s_x \cdot a(x)(1 + b(x)) |s_x|^{p-2} s_x
		+ t_x \cdot a(x)(1 + b(x)) |t_x|^{p-2} t_x \right)
		\\
		& \quad =
		s_x\bar{s} + t_x\bar{t} 
		- \frac{a(x)}{p} \left( |s_x|^p + 2b(x) |s_x|^\frac{p}{2} |t_x|^\frac{p}{2} + |t_x|^p \right)
		\\
		& \quad \leq
		\sup_{s, t \in \R} \left(
		s \bar{s} + t \bar{t} 
		- \frac{a(x)}{p} \left( |s|^p + 2b(x) |s|^\frac{p}{2} |t|^\frac{p}{2} + |t|^p \right)
		\right)
		\\
		& \quad = h(x, \bar{s}, \bar{t}).
\end{align*}
\end{enumerate}
\qed

\subsection{Proof of equation~\eqref{eq_auxiliary_h-below}}
\textit{
	Recall that $h_+$ was defined to be the Legendre transform of $f_+$ and
	$f_+(s, t) = \frac{1}{p} \left( |s|^p + 2 b_+ \: |s|^\frac{p}{2} |t|^\frac{p}{2} + |t|^p \right)$.
	For $\bar{u}, \bar{v} \in \L{p^\prime}$, we prove the inequality
	\begin{align*}
		\int_{\R^N} h(x, \bar{u}, \bar{v}) \: \mathrm{d}x 
		\geq h_+\left(\norm{a^{-1/p} \, \bar{u}}_{p^\prime}, 
		\norm{a^{-1/p} \, \bar{v}}_{p^\prime}\right).
	\end{align*}
}

	By definition of the Legendre transform, we have for $x \in \R^N$
	\begin{align*}
		h(x, \bar{u}(x), \bar{v}(x)) &= \sup_{s, t \in \R} \left[ s \bar{u}(x) + t \bar{v}(x) - f(x, s, t) \right].
	\end{align*}		
	In order to estimate the supremum, we insert explicitly
	\begin{align*}
		s_x &:= \frac{\sigma}{\norm{a^{-1/p} \, \bar{u}}_{p^\prime}^{p^\prime - 1}} 
		\cdot a(x)^{1 - p^\prime} |\bar{u}(x)|^{p^\prime - 2} \bar{u}(x),
		\\	
		t_x &:= \frac{\tau}{\norm{a^{-1/p} \, \bar{v}}_{p^\prime}^{p^\prime - 1}} 
		\cdot a(x)^{1 - p^\prime} |\bar{v}(x)|^{p^\prime - 2} \bar{v}(x)			
	\end{align*}
	where $\sigma, \tau \in \R$ are arbitrary. 
	With that, we integrate, estimate $b(x) \leq b_+$ and apply Hölder's inequality:
	\begin{align*}
		&\int_{\R^N} h(x, \bar{u}(x), \bar{v}(x)) \: \mathrm{d}x
		\\
		&\quad \geq
		\int_{\R^N} s_x \bar{u}(x) + t_x \bar{v}(x) - 
		\frac{a(x)}{p} \left( |s_x|^p + 2b(x) |s_x|^\frac{p}{2} |t_x|^\frac{p}{2} + |t_x|^p \right) \: \mathrm{d}x
		\\		
		&\quad  =
		\sigma \norm{a^{-1/p} \, \bar{u}}_{p^\prime}
		+
		\tau \norm{a^{-1/p} \, \bar{v}}_{p^\prime}
		\\*
		& \quad \quad
		- \frac{1}{p} \left(
		|\sigma|^p + 
		2 b_+ \frac{|\sigma \tau|^\frac{p}{2}}
		{\norm{a^{-1/p} \, \bar{u}}_{p^\prime}^\frac{p^\prime}{2} 
		\norm{a^{-1/p} \, \bar{v}}_{p^\prime}^\frac{p^\prime}{2}}
		\cdot
		\int_{\R^N}
		\left(a(x)^{-\frac{1}{p}}|\bar{u}|\right)^\frac{p^\prime}{2}
		\left(a(x)^{-\frac{1}{p}}|\bar{v}|\right)^\frac{p^\prime}{2}
		\: \mathrm{d}x
		+ |\tau|^p
		\right)
		\\
		& \quad \geq
		\sigma \norm{a^{-1/p} \, \bar{u}}_{p^\prime}
		+
		\tau \norm{a^{-1/p} \, \bar{v}}_{p^\prime}
		- \frac{1}{p} \left(
		|\sigma|^p + 
		2 b_+ |\sigma \tau|^\frac{p}{2}
		+ |\tau|^p
		\right)
		\\
	 & \quad	=
		\sigma \norm{a^{-1/p} \, \bar{u}}_{p^\prime}
		+
		\tau \norm{a^{-1/p} \, \bar{v}}_{p^\prime}
		- f_+(\sigma, \tau).
	\end{align*}
	Passing to the supremum with respect to $\sigma, \tau \in \R$, we find the asserted inequality.
\qed	
	
\subsection{Proof of equation~\eqref{eq_auxiliary_mini}}

\textit{
Let $p > 4$ and consider, for $\eta \geq 0$, 
\begin{align*}
	\psi(\eta) := \frac{(1 + \eta^2)^\frac{1}{2}}{(1 + (2^\frac{p}{2} - 2) \eta^\frac{p}{2} + \eta^p)^\frac{1}{p}}.
\end{align*}
We assert that $\psi$ has exactly three critical points on $(0, \infty)$ which are given by 
$\{ \eta_1, 1, \eta_1^{-1} \}$ for some $\eta_1 \in (0, 1)$, and that $\psi$ attains its minimum on $(0, \infty)$
uniquely at $\eta_0 = 1$. 
}
\absatz

	We note that $\psi$ is smooth on $(0, \infty)$, 
	and that $\psi(\eta) \to 1$ as $\eta \searrow 0$ or $\eta \nearrow \infty$. Moreover, 
	$\psi(\eta^{-1}) = \psi(\eta)$ holds for all $\eta > 0$. 
	Critical points of $\psi$ satisfy 
	\begin{subequations}
	\begin{align}\label{eq_append_psicrit_a}
		0 = \psi'(\eta), 
		\text{   equivalently   }	
		1 + (2^\frac{p-2}{2} - 1) \eta^\frac{p}{2} = \eta^{p-2} + (2^\frac{p-2}{2} - 1) \eta^\frac{p-4}{2}.
	\end{align}	  
	Obviously, this is satisfied for $\eta_0 := 1$. Moreover, $p > 4$ implies that
	$\psi''(1) = \frac{2^\frac{p}{2} - p}{2 \cdot 2^\frac{p}{2}} > 0$, 
	which proves that $\psi(1) = 1$ is a strict local minimum.
	Once we have established that $\psi$ has a unique critical point $\eta_1$ in the interval $(0, 1)$, 
	we colclude that $\psi$ attains local maxima at $\eta_1$ and at $\eta_1^{-1}$ 
	and hence that the local minimum in $\eta_0 = 1$ is in fact global.
	
	We substitute $\kappa := 2^\frac{p-2}{2} - 1 (> 1)$, 
	$\sigma := \frac{p-4}{p} \in (0, 1)$, $y := \eta^\frac{p}{2}$ and~\eqref{eq_append_psicrit_a} gives
	\begin{align}\label{eq_append_psicrit_b}
		0 = \psi'\left( y^\frac{2}{p} \right)
		\quad \Leftrightarrow \quad 
		\frac{1 + \kappa y}{\kappa + y} = y^\sigma.
	\end{align}
	\end{subequations}
	
	\textit{Existence of $\eta_1$}:
	This is guaranteed by the Mean Value Theorem since $\psi (0) = \psi (1) = 1$.
	
	\textit{Uniqueness of $\eta_1$}:
	Now assume that $\psi$ possesses (at least) two critical points $\eta_1, \eta_2$ in $(0, 1)$ 
	with $0 < \eta_1 < \eta_2 < 1$; 
	then $\frac{1}{\eta_2}, \frac{1}{\eta_1} \in (1, \infty)$ are two more critical points. 
	We denote $y_j := \eta_j^\frac{p}{2}$ for $j = 0, 1, 2$. Notice that, by~\eqref{eq_append_psicrit_b}, 
	we have
	\begin{align*}
		\frac{1 + \kappa y}{\kappa + y} - y^\sigma = 0
		\quad
		\text{for }
		y \in \left\{ y_1, y_2, 1, \frac{1}{y_1}, \frac{1}{y_2} \right\}.
	\end{align*}
	The Mean Value Theorem yields 
	$z_1 \in (y_1, y_2), z_2 \in (y_2, 1), 
	z_3 \in \left(1, \frac{1}{y_2}\right), z_4 \in \left(\frac{1}{y_2}, \frac{1}{y_1}\right)$ with
	\begin{align*}
		\frac{\mathrm{d}}{\mathrm{d}y}{\bigg|_{y = z_j}} 
		\left( \frac{1 + \kappa y}{\kappa + y} -  y^\sigma \right) = 0,
		\quad \text{equivalently} \quad  
		\sqrt{\frac{\sigma}{\kappa^2 - 1}} (\kappa + z_j) = z_j^{\frac{1 - \sigma}{2}}.
	\end{align*}	  
	Then again, we find $z^\ast_1 \in (z_1, z_2)$ and $z^\ast_2 \in (z_3, z_4)$ satisfying
	\begin{align*}
		\frac{\mathrm{d}}{\mathrm{d}y}{\bigg|_{y = z_j^\ast}}
		 \left(\sqrt{\frac{\sigma}{\kappa^2 - 1}} (\kappa + y) -  y^{\frac{1 - \sigma}{2}} \right) = 0,
		\quad \text{equivalently} \quad  
		\frac{(1 - \sigma)^2 ( \kappa^2 - 1)}{4\sigma} = (z_j^\ast)^{\sigma + 1}.
	\end{align*}
	The latter equation, however, possesses a unique positive solution;
	since we have found two distinct ones $z_1^\ast \in (0, 1), z_2^\ast \in (1, \infty)$, we have a
	contradiction. 

\qed

\section*{Acknowledgements}
The authors thank Gilles Ev\'{e}quoz (University of Frankfurt) for stimulating discussions and explanations about the subject as well as for the improvements explained in Remark~\ref{rmk_restriction}. They would also like to express their gratitude to the reviewer for thorough revision and helpful comments. Both authors gratefully acknowledge financial support by the Deutsche Forschungs\-gemeinschaft (DFG) through CRC 1173 ''Wave phenomena: analysis and numerics''. 

\bibliographystyle{abbrv}	
\bibliography{MandelScheider_Helmholtz_Literatur} 
  
\end{document}